\documentclass{amsart}

\usepackage{amsmath, amssymb, amsthm,enumitem,bm,pict2e,xcolor,hyperref,tikz-cd}
\usepackage{indentfirst}

\hypersetup{
    colorlinks=true,
    linkcolor=blue,
    filecolor=magenta,      
    urlcolor=cyan,
}

\theoremstyle{definition}
\newtheorem{theorem}{Theorem}
\newtheorem*{theorem*}{Theorem}
\numberwithin{theorem}{section}

\newtheorem{proposition}[theorem]{Proposition}
\newtheorem{lemma}[theorem]{Lemma}
\newtheorem{remark}[theorem]{Remark}
\newtheorem{example}[theorem]{Example}
\newtheorem{cor}[theorem]{Corollary}

\newtheorem{innercustomthm}{Theorem}
\newenvironment{customthm}[1]
  {\renewcommand\theinnercustomthm{#1}\innercustomthm}
  {\endinnercustomthm}

\DeclareMathOperator{\grad}{grad}
\DeclareMathOperator{\spann}{span}

\DeclareMathOperator{\initial}{in}

\DeclareMathOperator{\Hom}{Hom}

\newcommand{\mult}{\mathrm{mult}}
\newcommand{\la}{\lambda}

\newcommand{\mL}{\mathcal L}
\newcommand{\mM}{\mathcal M}

\newcommand{\mP}{\mathcal P}

\newcommand{\bC}{\mathbb{C}}
\newcommand{\bP}{\mathbb{P}}
\newcommand{\bR}{\mathbb{R}}
\newcommand{\bZ}{\mathbb{Z}}

\mathcode`l="8000
\begingroup
\makeatletter
\lccode`\~=`\l
\DeclareMathSymbol{\lsb@l}{\mathalpha}{letters}{`l}
\lowercase{\gdef~{\ifnum\the\mathgroup=\m@ne \ell \else \lsb@l \fi}}%
\endgroup

\title{Semitoric degenerations of Hibi varieties and flag varieties}

\author{Evgeny Feigin}
\address{E. Feigin:\newline
National Research University Higher School of Economics\\
Faculty of Mathematics\\
Ulitsa Usacheva 6\\Moscow 119048\\Russia\newline
{\it and }\newline
Skolkovo Institute of Science and Technology\\ 
Center for Advanced Studies\\
Bolshoy Boulevard 30, bld. 1\\
Moscow 121205\\
Russia
}
\email{evgfeig@gmail.com}
\author{Igor Makhlin}
\address{I. Makhlin:\newline
Skolkovo Institute of Science and Technology\\ 
Center for Advanced Studies\\
Bolshoy Boulevard 30, bld. 1\\
Moscow 121205\\
Russia\newline
{\it and }\newline
National Research University Higher School of Economics\\
Faculty of Mathematics\\
Ulitsa Usacheva 6\\Moscow 119048\\Russia}
\email{imakhlin@mail.ru}

\dedicatory{Dedicated to the memory of our teacher and colleague Ernest Borisovich Vinberg.}

\thanks{Both authors were partially supported by the grant RSF 19-11-00056. 
I. Makhlin was supported in part by the Young Russian Mathematics award.}

\begin{document}

\begin{center}
\textcolor{red}{\Large After this paper was written the authors found the work~\cite{Zh} which proves a general theorem relating degenerations of toric varieties to regular subdivisions of polytopes. This implies some of the results found below. Subsequently, we decided to rewrite this paper in a way that would make use of Zhu's theorem to substantially generalize our construction: see~\cite{FM}.}
\vspace{1cm}
\end{center}

\maketitle

\begin{abstract} 
We first construct a family of flat semitoric degenerations for the Hibi variety of every finite distributive lattice. The irreducible components of each degeneration are the toric varieties associated with polytopes forming a regular subdivision of the order polytope of the underlying poset. These components are themselves Hibi varieties. For each degeneration in our family we also define the corresponding weight polytope and embed the degeneration into the associated toric variety as the union of orbit closures given by a set of faces. Every such weight polytope projects onto the order polytope with the chosen faces projecting into the parts of the corresponding regular subdivision. We then apply these constructions to obtain a family of flat semitoric degenerations for every type A Grassmannian and complete flag variety.
\end{abstract}


\section*{Introduction}

The construction and study of flat degenerations of Lie theoretic varieties such as flag and Schubert varieties has been a popular subject for the last three decades. Historical overviews of the results obtained in this field can be found in~\cite{Kn} and~\cite{FaFL}. The latter text considers a particularly popular subgenre where the flat degeneration is a toric variety. This allows one to apply the powerful machinery of toric geometry to the study of Lie theory. A natural development of this idea is to consider \emph{semitoric} degenerations: those with toric irreducible components. Papers obtaining degenerations of this type include~\cite{Ch,KM,MG}.

This paper constructs a new family of flat semitoric degenerations of type A flag varieties. More specifically, we construct a family of semitoric varieties for every Hibi variety and then apply this construction to Grassmannians and complete flag varieties. The Hibi variety $H$ (also referred to as the ``Hibi toric variety'') is a projective toric variety associated with every finite distributive lattice $\mL$. It is the subvariety in $\bP(\bC^\mL)$ cut out by the Hibi ideal $I^h$ in the ring $R=\bC[\{X_a,a\in\mL\}]$ generated by the expressions \[X_aX_b-X_{a\land b}X_{a\lor b}.\] These varieties owe their name to the paper~\cite{H} and were studied in works of various authors such as \cite{GL,BL,FL}.

It is not hard to see (\cite{BL,FL}) that $H$ has a flat degeneration cut out by the ideal $I^m\subset R$ generated by products $X_aX_b$ for incomparable $a$, $b$. In fact, this is a Gr\"obner degeneration of $H$, meaning that the ideal $I^m$ is an initial ideal of $I^h$. One may consider the intermediate degenerations: those cut out by ideals which are initial ideals of $I^h$ and have $I^m$ as an initial ideal. This is precisely the family of degenerations studied in this paper. Note that since $I^m$ is a monomial ideal, its zero set is a union of projective spaces which can be viewed as a simple example of a semitoric variety. We show that all the intermediate degenerations are also semitoric and describe them explicitly in two different ways.

Before presenting our first description let us recall several key notions which are covered in detail in Section~\ref{preliminaries}. Every initial ideal of $I^h$ can be obtained by choosing a point $w\in\bR^\mL$, defining a grading on $R$ by $\grad_w(X_a)=w_a$ and considering the initial ideal $\initial_{\grad_w} I^h$. The set of $w\in\bR^\mL$ with a given initial ideal $\initial_{\grad_w} I^h$ is a polyhedral cone and together these cones compose the Gr\"obner fan of $I$. The monomial ideal $I^m$ corresponds to a maximal cone $K$ in the Gr\"obner fan, this maximal cone is described explicitly in~\cite{M}. The faces of $K$ enumerate the intermediate degenerations studied here. We denote by $I^F$ the initial ideal corresponding to a face $F\subset K$ and by $Y^F\subset\bP(\bC^\mL)$ the zero set of $I^F$.

Another important attribute of the lattice $\mL$ is its poset of join-irreducible elements $(P,<)$. The Hibi variety $H$ is known to be the toric variety associated with the order polytope $O(P,<)\subset\bR^P$ of this poset (\cite{stan}). The vertices of $O(P,<)$ are the indicator functions of order ideals in $(P,<)$ and are, therefore, enumerated by $\mL$. This means that every $w\in\bR^\mL$ defines a regular polyhedral subdivision $\Theta(w)$ (\cite{GKZ}) of $O(P,<)$. The main result of Section~\ref{semitoric} is as follows.
\begin{customthm}{A}[cf. Theorem~\ref{main}]\label{thma}
Choose a face $F$ of $K$ and a point $w$ in the relative interior of $F$. The regular subdivision $\Theta(w)$ is independent of the choice of $w$ and every part in the subdivision is an order polytope of some order relation on $P$. The irreducible components of the degeneration $Y^F$ are the toric varieties associated with the parts of the subdivision $\Theta(w)$ (and are, therefore, themselves Hibi varieties). 
\end{customthm}
Moreover, this correspondence between points in $K$ and regular subdivisions identifies $K$ with a maximal cone in the secondary fan of $O(P,<)$, see Corollary~\ref{maxcone}. In the extreme case of $F=K$ the regular subdivision is a triangulation of $O(P,<)$ into simplices indexed by the maximal chains in $\mL$. We note that this triangulation is discussed in the paper~\cite{FL}.

In Section~\ref{weightpolytopes} we give the second description, to do so we introduce the notion of weight polytopes. For every face $F$ of $K$ and every $a\in\mL$ the coordinate in $\bR^\mL$ corresponding to $a$ can be viewed as a functional $\lambda^F_a$ on $\spann(F)$. The weight polytope $\Pi^F\subset\spann(F)^*$ is defined as the convex hull of the $\la^F_a$. This polytope arises naturally as the weight diagram of the linear component of $R/I^F$ under the action of a certain torus. The polytope $\Pi^K$ corresponding to the maximal face is a $|\mL|-1$ dimensional simplex while the polytope $\Pi^{F_0}$ corresponding to the apex $F_0\subset K$ is identified with the order polytope $O(P,<)$. Together the polytopes $\Pi^F$ form a projective system indexed by the face lattice of $K$: one has a linear surjection $\pi^G_F:\Pi^G\to\Pi^F$ for every pair of faces $F\subset G$. 
\begin{customthm}{B}[cf. Theorem~\ref{main2}]\label{thmb}
Choose a face $F$ of $K$. The polytope $\Pi^F$ has a distinguished set of $|P|$-dimensional faces $\Phi_1^F,\dots,\Phi_{m(F)}^F$ which are mapped bijectively by $\pi^F_{F_0}$ into $\Pi^{F_0}=O(P,<)$. The images $\pi^F_{F_0}(\Phi_1^F),\dots,\pi^F_{F_0}(\Phi_{m(F)}^F)$ are precisely the parts in the subdivision considered in Theorem~\ref{thma}. The degeneration $Y^F$ is embedded into the toric variety associated with $\Pi^F$ as the union of toric subvarieties corresponding to the faces $\Phi_1^F,\dots,\Phi_{m(F)}^F$.
\end{customthm}

In particular, Theorem~\ref{thmb} shows that $Y^F$ is not simply a union of toric varieties but is a union of torus orbit closures inside a larger toric variety. This is a property shared by other semitoric degenerations considered in the literature such as those found in~\cite{Ch,KM,MG}.   

In the final Section~\ref{flags} we apply the above constructions to the study of flag varieties. First Theorems~\ref{thma} and~\ref{thmb} are applied directly to obtain a family of flat semitoric degenerations for every type A Grassmannian $\mathrm{Gr}_k(n)$. This is done by specializing to a certain lattice $\mL$ indexed by $k$-subsets in $\{1,\dots,n\}$. In this case the varieties cut out by the ideals $I^h$ and $I^m$ are well known to be flat degenerations of $\mathrm{Gr}_k(n)$: the case of $I^h$ is due to~\cite{GL} while the case of $I^m$ can be traced to the work of W. V. D. Hodge. This means that all the intermediate degenerations are also flat degenerations of $\mathrm{Gr}_k(n)$ with their structure described by Theorems~\ref{thma} and~\ref{thmb}.

The case of the variety $\mathcal F_n$ of complete flags requires somewhat more care. Here one considers a specific lattice $\mL$ indexed by nonempty proper subsets of $\{1,\dots,n\}$. The paper \cite{GL} shows that the corresponding Hibi ideal $I^h$ also defines a flat degeneration of $\mathcal F_n$, however, in a multiprojective sense: one views $R$ as the multihomogeneous coordinate ring of \[\bP(\wedge^1\bC^n)\times\dots\times\bP(\wedge^{n-1}\bC^n)\] and considers the zero set in this product. The obtained degeneration is again toric, the associated polytope is the Gelfand--Tsetlin polytope (as shown in~\cite{KM}). The Gelfand--Tsetlin polytope is known to be a marked order polytope (\cite{ABS}) for the underlying poset $(P,<)$. The task here is to adjust Theorem~\ref{thma} to this multiprojective and marked setting.
\begin{customthm}{C}[cf. Theorem~\ref{mainflags}]
Consider a face $F$ of $K$ and let $Y^F_\mult$ be the multiprojective variety cut out by $I^F$. Then $Y^F$ is a flat degeneration of $\mathcal F_n$. Its irreducible components are the toric varieties associated with marked order polytopes of certain orders on $P$. Together these polytopes form a regular subdivision of the Gelfand--Tsetlin polytope.
\end{customthm}

To avoid overloading the paper we do not explicitly prove a version of Theorem~\ref{thmb} for this case. Nevertheless, we believe that such a version does exist. We briefly define the relevant objects and write what the corresponding statements should be while only partially outlining a proof.


\section{Preliminaries}\label{preliminaries}

\subsection{Gr\"obner degenerations and Gr\"obner fans}

For a positive integer $N$ consider the polynomial ring $R=\bC[X_1,\ldots,X_N]$. For $d=(d_1,\ldots,d_N)\in\bZ_{\ge0}^N$ let $\mathbf X^d=\prod X_i^{d_i}$ and let $(,)$ be the standard scalar product in $\bR^N$. For a set of vectors $d^j\in\bZ_{\ge0}^N$ consider $p=\sum_j c_j\mathbf X^{d^j}\in R$ with $c_j\neq 0$. Consider a vector $w\in\bR^N$ and let $\min_j (w,d^j)=m$. The corresponding initial part of $p$ is then \[\initial_w p=\sum_{j|(w,d^j)=m} c_j\mathbf X^{d^j}.\] In other words, we define a grading on $R$ by setting the grading of $X_i$ equal to $w_i$ and then take the nonzero homogeneous component of $p$ of the least possible grading. For an ideal $I\subset R$ its initial ideal $\initial_w I$ is the linear span of $\{\initial_w p, p\in I\}$ which is easily seen to be an ideal in $R$.

An important property of this construction is that the algebra $R/\initial_w I$ is a \emph{flat degeneration} of the algebra $R/I$, i.e.\ there exists a flat $\bC$-family of algebras with all general fibres isomorphic to $R/I$ and the special fibre isomorphic to $R/\initial_w I$.
\begin{theorem}[see, for instance,~{\cite[Corollary 3.2.6]{HH}}]\label{flatfamily}
There exists a flat $\bC[t]$-algebra $S$ such that $S/\langle t-a\rangle\simeq R/I$ for all nonzero $a\in\bC$ while $S/\langle t\rangle\simeq R/\initial_w I$.
\end{theorem}

Now let $I$ be homogeneous with respect to the standard grading by total degree and let $\initial_w I$ be a radical ideal. The former obviously implies that $\initial_w I$ is also homogeneous while the latter implies that $I$ is itself radical (see~\cite[Proposition 3.3.7]{HH}). In this case $I$ is the vanishing ideal of a variety $X\subset\bP(\bC^N)$ and $\initial_w I$ is the vanishing ideal of a variety $X^w\subset\bP(\bC^N)$. We obtain the following geometric reformulation of the above theorem.
\begin{cor}\label{gdegen}
There exists a flat family $\mathcal X\subset\bP(\bC^N)\times\bC$ over $\bC$ such that for the projection $\pi$ onto $\bC$ any general fibre $\pi^{-1}(a)$ with $a\neq 0$ is isomorphic to $X$ while the special fibre $\pi^{-1}(0)$ is isomorphic to $X^w$.
\end{cor}
This corollary states that $X^w$ is a flat degeneration of $X$, a flat degeneration of this form is known as a \emph{Gr\"obner degeneration}.

We now move on to define the Gr\"obner fan of $I$ which parametrizes its initial ideals. We retain the assumption that $I$ is homogeneous but $I$ need not be radical here. 

For an ideal $J\subset R$ denote $C(I,J)\subset\bR^N$ the set of points $w$ for which $\initial_w I=J$. The nonempty sets $C(I,J)$ form a partition of $\bR^N$ with $w$ contained in $C(I,\initial_w I)$. This partition is known as the \emph{Gr\"obner fan} of $I$ (introduced in~\cite{MR}), its basic properties are summed up in the below theorem. This information can be found in Chapters 1 and 2 of~\cite{S} (modulo a switch between the $\min$ and $\max$ conventions). 
\begin{theorem}\label{gfans}
\hfill
\begin{enumerate}[label=(\alph*)]
\item There are only finitely many different nonempty sets $C(I,J)$.
\item Every nonempty $C(I,J)$ is a relatively open polyhedral cone.
\item Together all the nonempty $C(I,J)$ form a polyhedral fan with support $\bR^N$. This means that every face of the closure $\overline{C(I,J)}$ is itself the closure of some $C(I,J')$. 
\item\label{IJ'J} If $\overline{C(I,J')}$ is a face of $\overline{C(I,J)}$, then the set $C(J',J)$ is nonempty. Conversely, if the sets $C(I,J')$ and $C(J',J)$ are both nonempty, then so is $C(I,J)$ and $\overline{C(I,J')}$ is a face of $\overline{C(I,J)}$.
\item A nonempty cone $C(I,J)$ is maximal in the Gr\"obner fan if and only if $J$ is monomial.
\end{enumerate}
\end{theorem}

\subsection{Hibi varieties and order polytopes}

Let $(\mL,\lor,\land)$ be a finite distributive lattice with induced order relation $<$ (so, for instance, $a=b\lor c$ is the $<$-minimal element such that $b<a$ and $c<a$). For each $a\in\mL$ introduce the variable $X_a$ and consider the polynomial ring $R(\mL)=\bC[\{X_a,a\in\mL\}]$. The \textit{Hibi ideal} of $\mL$ is the ideal $I^h(\mL)\subset R(\mL)$ generated by the elements \[d(a,b)=X_aX_b-X_{a\lor b}X_{a\land b}\] for all $a,b\in\mathcal L$, note that $d(a,b)\neq 0$ only if $a$ and $b$ are $<$-incomparable. This notion originates from~\cite{H}.

The corresponding \emph{Hibi variety} is the subvariety $H(\mL)\subset\bP(\bC^\mL)$ cut out by $I^h(\mL)$. It turns out that this variety is the toric variety associated with the order polytope of the poset of join-irreducible elements in $\mL$. Below we define these notions.

First we will need the fundamental theorem of finite distributive lattices. An element $a\in\mL$ is called join-irreducible if $a$ is not minimal in $\mL$ and $a=b\lor c$ implies $a=b$ or $a=c$. This is equivalent to $a$ covering exactly one element. Let $P(\mL)$ denote the set of join-irreducible elements in $\mL$, we have the poset $(P(\mL),<)$ with the induced order relation. 

A small remark on notation: when speaking of distributive lattices we will often use the underlying set to denote the whole lattice, e.g.\ write $\mL$ instead of $(\mL,\lor,\land)$ or $(\mL,<)$. This will be unambiguous because we will not be considering two different distributive lattice structures on the same set. However, when speaking of general posets $(P,<)$ (which appear in this paper as posets of join-irreducible elements in lattices) we will always differentiate between $(P,<)$ and $P$, since we will be considering multiple order relations on the same set.  

Now, for a poset $(P,<)$ let $\mathcal J(P,<)$ be the set of order ideals in $(P,<)$ (we use the term ``order ideal'' synonymously to ``lower set''). It is easy to see that $\mathcal J(P,<)$ is a distributive lattice with union (of order ideals) as join, intersection as meet and inclusion as the order relation. The following classical result due to Garret Birkhoff is known as \emph{Birkhoff's representation theorem} or the \emph{fundamental theorem of finite distributive lattices}, a proof can be found in~\cite[Theorem 9.1.7]{HH}.
\begin{theorem}
The distributive lattices $\mL$ and $\mathcal J(P(\mL),<)$ are isomorphic. An isomorphism $\iota_\mL$ is obtained by mapping $a\in\mL$ to the order ideal in $P(\mL)$ composed of all join-irreducible elements $p$ with $p\le a$.  
\end{theorem}

In the next proposition we list some basic properties of this correspondence that are straightforward from the definitions. Recall that the height of $\mL$ is the number of elements in any maximal chain in $\mL$ and that the height $|a|$ of $a\in\mL$ is the distance in the Hasse diagram between $a$ and the minimal element in $\mL$.
\begin{proposition}\label{basic}
\hfill
\begin{enumerate}[label=(\alph*)]
\item $\iota_\mL^{-1}(\varnothing)$ is the unique minimal element in $\mL$ and $\iota_\mL^{-1}(P)$ is the unique maximal element.
\item $a$ covers $b$ in $\mL$ if and only if $\iota_\mL(a)$ is obtained from $\iota_\mL(b)$ by adding a single element, the cardinality $|\iota_\mL(a)|$ is equal to the height $|a|$. 
\item The height of $\mL$ is equal to $|P|+1$.
\item $p\in\mL$ is join-irreducible if and only if the order ideal $\iota_\mL(p)$ is principal: it consists of all $q\in P(\mL)$ with $q\le p$.
\item Every $a\in\mL$ is equal to the join of all join-irreducible elements $p\in\iota_\mL(a)$.
\end{enumerate}
\end{proposition}

We have the following alternative interpretation of the Hibi ideal. Consider variables $z_p$ indexed by $p\in P(\mL)$ and let $S=\bC[\{z_p\},t]$. Define a homomorphism $\varphi:R(\mL)\to S$ given by \[\varphi(X_a)=t\prod_{p\in\iota_\mL(a)}z_p.\] It is immediate from the fundamental theorem that the generators $d(a,b)$ of $I^h(\mL)$ lie in the kernel of $\varphi$. Moreover, the following holds.
\begin{proposition}[{\cite[Section 2]{H}}]\label{kernel}
$I^h(\mL)$ is the kernel of $\varphi$.
\end{proposition}

Next, for an arbitrary finite poset $(P,<)$ we define its \emph{order polytope} $O(P,<)$. This is a convex polytope in the space $\bR^P$ consisting of points $x$ with coordinates $(x_p, p\in P)$ satisfying $0\le x_p\le 1$ for all $p$ and $x_p\le x_q$ whenever $p>q$. It is immediate from the definition that the set of integer points in $O(P,<)$ is the set of indicator functions $\bm 1_J$ of order ideals $J\subset P$. It is also easily seen that each of these integer points is a vertex of $O(P,<)$ (see~\cite[Corollary 1.3]{stan}). This shows that the vertices (and integer points) of $O(P(\mL),<)$ are in natural bijection with $\mathcal J(P(\mL),<)$ and, via the fundamental theorem, with $\mL$.

It should be pointed out that the original definition in~\cite{stan} contains the reverse (and, perhaps, more natural) inequality $x_p\le x_q$ whenever $p<q$, this amounts to reflecting $O(P,<)$ in the point $(\frac12,\dots,\frac12)$. We use the above definition to adhere to the standard convention of considering join-irreducibles and order ideals in the fundamental theorem (rather than meet-irreducibles and order filters).

The order polytope $O(P(\mL),<)$ has the important property of being \emph{normal} (see, for instance,~\cite[Theorem 2.5]{FF}). This means that for any integer $k>0$ every integer point in its dilation $kO(P(\mL),<)$ can be expressed as the sum of $k$ (not necessarily distinct) integer points in $O(P(\mL),<)$. In other words, this means that the set $kO(P(\mL),<)\cap\bZ^\mL$ is the $k$-fold Minkowski sum of $O(P(\mL),<)\cap\bZ^\mL$ with itself.

The toric variety associated with a normal polytope is easy to describe, its homogeneous coordinate ring is the semigroup ring of the semigroup generated by the polytope's integer points, see~\cite[\textsection2.3]{CLS}. In the case of $O(P(\mL),<)$ this ring is precisely the image of the map $\varphi$ considered in Proposition~\ref{kernel}. This leads to the following fact which is also sometimes attributed to~\cite{H}.
\begin{proposition}\label{toricvariety}
$H(\mL)$ is isomorphic to the toric variety associated with the polytope $O(P(\mL),<)$.
\end{proposition}

\subsection{The monomial ideal and the maximal cone}

As before, let $\mL$ be a finite distributive lattice. We consider the monomial ideal $I^m(\mL)\subset R(\mL)$ generated by products $X_aX_b$ for all incomparable pairs $\{a,b\}\subset\mathcal L$. Call a monomial $X_{a_1}\ldots X_{a_k}$ \textit{standard} if the elements $a_1,\ldots,a_k$ are pairwise comparable (they form a weak chain). Then one sees that non-standard monomials form a basis in $I^m(\mL)$. 

It is easily seen that $I^m(\mL)$ is an initial ideal of $I^h(\mL)$, i.e.\ the cone $C(I^h(\mL),I^m(\mL))$ is nonempty (for instance, see~\cite[Proposition 2.1]{M}). By part (e) of Theorem~\ref{gfans}, the cone $C(I^h(\mL),I^m(\mL))$ is maximal in the Gr\"obner fan of $I$, the paper~\cite{M} provides an explicit description of this maximal cone.

We say that $\{a,b\}\subset\mL$ is a \emph{diamond pair} if $a\lor b$ covers both $a$ and $b$ (in terms of the order $<$ on $\mL$) or, equivalently, both $a$ and $b$ cover $a\land b$. This means that in the Hasse diagram of the poset $(\mL,<)$ (we view Hasse diagrams as abstract oriented graphs with edges directed from the lesser vertex to the greater) we have the following ``diamond'' as a subgraph:

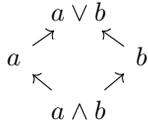
\begin{center}
\begin{tikzcd}[row sep=1mm, column sep=1mm]
&a\lor b&\\
 a  \arrow[ru]&& b\arrow[lu]\\[3pt]
 &\arrow[lu] a\land b\arrow[ru]&
\end{tikzcd}
\end{center}

Note that by part (b) of Proposition~\ref{basic} for a diamond pair $\{a,b\}$ we have \[|a|=|b|=|a\land b|+1=|a\lor b|-1.\] The same part (b) shows that order ideals $\iota_\mL(a)$ and $\iota_\mL(b)$ are both obtained from $\iota_\mL(a\land b)$ by adding one element and $\iota_\mL(a\lor b)$ is obtained from $\iota_\mL(a\land b)$ by adding both of these elements. Now to the minimal H-description.

\begin{theorem}[{\cite[Theorem 2.3]{M}}]\label{hibimonfacets}
The cone $C(I^h(\mL),I^m(\mL))$ consists of those $w\in\bR^\mL$ that satisfy \[w_a+w_b<w_{a\land b}+w_{a\lor b}\] for all diamond pairs $\{a,b\}$ in $\mL$. This H-description is minimal: for every diamond pair $\{a,b\}$ points in $\overline{C(I^h(\mL),I^m(\mL))}$ satisfying $w_a+w_b=w_{a\land b}+w_{a\lor b}$ form a facet of $\overline{C(I^h(\mL),I^m(\mL))}$.
\end{theorem}

\subsection{Regular subdivisions and secondary fans}\label{secondary}

The notions and results here are due to Gelfand, Kapranov and Zelevinsky, see~\cite[Chapter 7]{GKZ}. They, however, use the terms ``coherent subdivision'' and ``coherent triangulation'' rather than ``regular subdivision'' and ``regular triangulation'' that are commonly used today.

Consider a convex polytope $Q\subset\bR^n$ of dimension $n$ with set of vertices $V=\{v_1,\dots,v_k\}$ and a point $c=(c_1,\dots,c_k)\in\bR^V$. Let $S\subset \bR^n\times\bR$ be the union of rays $v_i\times\{x\le c_i\}$ for $i\in[1,k]$ and let $T$ be the convex hull of $S$. It is evident that $T$ is an $(n+1)$-dimensional convex polyhedron with $\rho(T)=Q$ where $\rho$ denotes the projection onto $\bR^n$. Furthermore, $T$ has two kinds of facets $F$: those with $\dim\rho(F)=n$ and those with $\dim\rho(F)=n-1$. Each of the former facets is bounded and is the convex hull of some subset of the points $v_k\times c_k$. Each of the latter facets is an unbounded convex hull of a union of rays in $S$. One could say that the bounded facets are the ones you see when you ``look at $T$ from above''. 

Let $F_1,\dots,F_m$ be the bounded facets of $T$ and denote $Q_i=\rho(F_i)$. The set $\{Q_1,\dots,Q_m\}$ is a polyhedral subdivision of $Q$ which we denote $\Theta_Q(c)$. This means that $\bigcup Q_i=Q$, all $\dim Q_i=n$ and the $Q_i$ together with all their faces form a polyhedral complex. A polyhedral subdivision of the form $\Theta_Q(c)$ is known as a \emph{regular subdivision}. 

Note that the union of the bounded facets $F_i$ is the graph of a continuous convex piecewise linear function on $Q$ and that the $Q_i$ are the (maximal) domains of linearity of this function. A remark on the terminology: we say that a function $f(x)$ is (piecewise) linear on a subset of a real space if it is (piecewise) of the form $(u,x)+b$, functions of the form $(u,x)$ will be referred to as ``functionals''.

Regular subdivisions of $Q$ are also parametrized by a polyhedral fan. For a regular subdivision $\Theta$ denote $C(\Theta)\subset\bR^V$ the set of $c$ such that $\Theta_Q(c)=\Theta$.
\begin{theorem}\label{secondaryfan}
\hfill
\begin{enumerate}[label=(\alph*)]
\item Every $C(\Theta)$ is a relatively open polyhedral cone, together these cones form a polyhedral fan with support $\bR^V$ known as the \emph{secondary fan} of $Q$. 
\item $\overline{C(\Theta')}$ is a face of $\overline{C(\Theta)}$ if and only if $\Theta$ is a refinement of $\Theta'$.
\item $C(\Theta)$ is a maximal cone in the secondary fan if and only if $\Theta$ is a triangulation of $Q$ (such triangulations are known as \emph{regular triangulations}).
\item $C(\{Q\})$ is the minimal cone in the secondary fan, it is a linear space of dimension $n+1$.
\end{enumerate}
\end{theorem}

\section{Semitoric degenerations from regular subdivisions}\label{semitoric}

We fix a finite distributive lattice $\mL$ and denote $R=R(\mL)$, $I^h=I^h(\mL)$, $I^m=I^m(\mL)$, $H=H(\mL)$, $P=P(\mL)$ and $\iota=\iota_\mL$. We also let $K$ denote the closure $\overline{C(I^h,I^m)}$. 

In view of part (d) of Theorem~\ref{gfans} the faces of $K$ enumerate all ideals $I$ such that $I$ is an initial ideal of $I^h$ and $I^m$ is an initial ideal of $I$ (i.e.\ the cones $C(I^h,I)$ and $C(I,I^m)$ are both nonempty). Choose a face $F$ of $K$ and denote by $I^F$ the corresponding ideal and by $Y^F$ the variety cut out by $I^F$ in $\bP=\bP(\bC^\mL)$. The monomial ideal $I^m$ is squarefree and thus radical, therefore, $I^F$ is also radical and $Y^F$ is a flat degeneration of $H$ as in Corollary~\ref{gdegen}.

Before we state our main theorem we make a few simple order-theoretic observations. Consider a partial order $<'$ on $P$ that is stronger than $<$: for any $p,q\in P$ with $p<q$ we have $p<'q$. For $a\in\mL$ let $v_a=\bm 1_{\iota(a)}\in\bR^P$ be the corresponding vertex of $O(P,<)$. Then the following hold.
\begin{proposition}\label{stronger}
\hfill
\begin{enumerate}[label=(\alph*)]
\item We have inclusions of order polytopes $O(P,<')\subset O(P,<)\subset\bR^P$ and sets of order ideals $\mathcal J(P,<')\subset\mathcal J(P,<)$.
\item The vertices of $O(P,<')$ are those $v_a$ for which $\iota(a)\in\mathcal J(P,<')$, in particular, every vertex of $O(P,<')$ is a vertex of $O(P,<)$.
\item The set $\iota^{-1}(\mathcal J(P,<'))$ is a sublattice in $\mL$.
\end{enumerate}
\end{proposition}
\begin{proof}
The inclusions in (a) are immediate from the definitions. Part (b) is just the characterization of the vertices of an order polytope as indicator functions of order ideals. Part (c) holds because the set of order ideals $J(P,<')$ is closed under intersection and union while $\iota$ is a lattice isomorphism from $(\mL,\land,\lor)$ and $(\mathcal J(P,<),\cap,\cup)$.
\end{proof}

Denote the sublattice $\iota^{-1}(\mathcal J(P,<'))$ by $\mM_{<'}$. Note that the map $\iota_{\mM_{<'}}$ coincides with the restriction of $\iota$ to $\mathcal J(P,<')$. We have the natural embedding $H(\mM_{<'})\subset\bP(\bC^{\mM_{<'}})\subset\bP$. Next, let $V$ be the set of vertices of $O(P,<)$. The bijection from $\mL$ to $V$ mapping $a$ to $v_a$ provides a linear isomorphism $\psi:\bR^\mL\to\bR^V$. We now state the main results of this section.
\begin{theorem}\label{main}
\hfill
\begin{enumerate}[label=(\alph*)]
\item For every face $F$ of $K$ there exists $m(F)\in\mathbb N$ together with order relations $<_1^F,\dots,<_{m(F)}^F$ on $P$ such that for any $w\in C(I^h,I^F)$ (relative interior of $F$) we have \[\Theta_{O(P,<)}(\psi(w))=\{O(P,<_1^F),\dots,O(P,<_{m(F)}^F)\}.\] 
\item For every face $F$ of $K$ the variety $Y^F$ is semitoric with irreducible components $H(\mM_{<_1^F}),\dots,H(\mM_{<_{m(F)}^F})\subset\bP$ for the orders $<_1^F,\dots,<_{m(F)}^F$ provided by part (a). 
\item The map sending $F$ to $\{O(P,<_1^F),\dots,O(P,<_{m(F)}^F)\}$ is a bijection between the faces of $K$ and regular subdivisions of $O(P,<)$ into order polytopes of partial orders on $P$.
\end{enumerate}
\end{theorem}

We proceed to prove the theorem in a series of statements. First, let $\prec$ be a linearization of $<$ (a linear order on $P$ that is stronger than $<$) given by $p_1\prec\dots\prec p_{|P|}$. The polytope $O(P,\prec)$ is a simplex given by 
\begin{equation}\label{simplex}
0\le x_{p_1}\le\dots\le x_{p_{|P|}}\le 1. 
\end{equation}
Suppose that the vertices of this simplex are $v_{a_1},\dots,v_{a_{|P|+1}}$ for some $a_1,\dots,a_{|P|+1}\in\mL$. We denote the subset $\{a_1,\dots,a_{|P|+1}\}=C_\prec$.
\begin{proposition}\label{maxchains}
\hfill
\begin{enumerate}[label=(\alph*)]
\item The simplices $O(P,\prec)$ where $\prec$ ranges over all linearizations of $<$ form a triangulation of $O(P,<)$.
\item For every linearization $\prec$ the set $C_\prec$ is a maximal chain in $\mL$ and every maximal chain has the form $C_\prec$ for some linearization $\prec$.
\item For two linearizations $\prec$ and $\prec'$ the simplices $O(P,\prec)$ and $O(P,\prec')$ have a common facet if and only if the symmetric difference of $C_\prec$ and $C_{\prec'}$ is a diamond pair.
\end{enumerate}
\end{proposition}
\begin{proof}
Every $x\in O(P,<)$ satisfies~\eqref{simplex} for some $\{p_1,\dots,p_{|P|}\}=P$ where the order $\prec$ given by $p_1\prec\dots\prec p_{|P|}$ is a linearization of $<$. This shows that the union of these simplices is indeed $O(P,<)$. It is also evident that if $x\in O(P,<)$ has pairwise distinct coordinates, then it is only contained in a single simplex $O(P,\prec)$. We see that the interiors of the simplices are pairwise disjoint. This means that they indeed form a triangulation, since each of them is a convex hull of a set vertices of $O(P,<)$. This proves part (a).

For part (b), let $\prec$ be given by $p_1\prec\dots\prec p_{|P|}$. Then $\mathcal J(P,\prec)$ consists of the sets $\{p_1,\dots,p_i\}$, these sets form a chain with respect to inclusion. This means that $\iota^{-1}(\mathcal J(P,\prec))=C_{\prec}$ is indeed a chain. The maximality of this chain follows from part (c) of Proposition~\ref{basic}. 

Conversely, consider a maximal chain $C\subset\mL$ consisting of $c_0<\dots<c_{|P|}$. Every difference $\iota(c_i)\backslash\iota(c_{i-1})$ consists of a single element $p_i$ by part (b) of Proposition~\ref{basic}. We then have $\iota(c_i)=\{p_1,\dots,p_i\}$ for all $i$. Let $\prec$ be the linear order given by $p_1\prec\dots\prec p_{|P|}$. Then $\mathcal J(P,\prec)$ consists of the sets $\{p_1,\dots,p_i\}$ so that \[C_\prec=\iota^{-1}(\mathcal J(P,\prec))=C\] as desired.

For part (c), $O(P,\prec)$ and $O(P,\prec')$ have a common facet if and only if the symmetric difference of their vertex sets has size 2. Now, whenever the symmetric difference of maximal chains $C_\prec$ and $C_{\prec'}$ consists of two elements $a$ and $b$, we have $|a|=|b|$, $|a\land b|=|a|-1$ and $|a\lor b|=|a|+1$, i.e.\ $\{a,b\}$ is a diamond pair.
\end{proof}

\begin{example}\label{extremal}
Before we continue, let us explain how parts (a) and (b) of Theorem~\ref{main} work in the two extremal cases when $F$ is all of $K$ and the apex (minimal face) of $K$.

First suppose that $F$ is the apex. Then we have $\initial_w I^h=I^h$ which means that $I^h$ is homogeneous with respect to the grading on $R$ that equals $w_a$ on $X_a$. It is not hard to check (see proof of Proposition~\ref{hibigradings}) that every such $w$ is obtained by choosing a linear function $f$ on $\bR^P$ and setting $w_a=f(v_a)$. Consequently, the regular subdivision $\Theta_{O(P,<)}(\psi(w))$ must consist of the single part $O(P,<)$, this is in accordance with part (a) of the theorem. The variety $Y^F$ is the Hibi variety $H$ which is indeed the toric variety of the polytope $O(P,<)$ as required by part (b).

Now let $F=K$. Then we have $\initial_w I^h=I^m$. One may then check (see proof of Theorem~\ref{main} below) that the subdivision $\Theta_{O(P,<)}(\psi(w))$ is the triangulation into the simplices $O(P,\prec)$ given by linearizations of $<$, this provides part (a). Now for a maximal chain $C$ in $\mL$ consider the ideal $I_C$ generated by the elements $X_a$ with $a\notin C$. The ideal $I^m$ is spanned by monomials $X_{a_1}\dots X_{a_N}$ such that some $a_i$ and $a_j$ are incomparable. Therefore, it is the intersection of the ideals $I_C$ with $C$ ranging over all maximal chains and these ideals cut out the irreducible components of $Y^F$. However, for a maximal chain $C_\prec$ corresponding to the linearization $\prec$ the ideal $I_{C_\prec}$ is seen to cut out the variety $H(\mM_\prec)\subset\bP$ (note that $\mM_\prec=C_\prec$ is a linearly ordered lattice with a trivial Hibi ideal). This variety is simply a $|P|$-dimensional projective space, so indeed the toric variety of the simplex.
\end{example}

\begin{lemma}\label{convex}
Let $Q$ be any convex polytope and let $Q_1,\dots,Q_n$ form a polyhedral subdivision of $Q$. Let $f:Q\to\bR$ be linear on every $Q_i$. Suppose that $f$ is convex on every union $Q_i\cup Q_j$ where $Q_i$ and $Q_j$ have a common facet. Then $f$ is convex on all of $Q$.
\end{lemma}
\begin{proof}
We are to show that for every pair of points $a,b\in Q$ all points of the segment $[a\times f(a),b\times f(b)]\subset Q\times\bR$ lie on or below the graph of $f$. Evidently, this condition is closed: it is sufficient to verify it for a set of pairs $(a,b)$ that is dense in $Q^2$. Therefore, we may assume that the segment $[a,b]$ does not intersect any face of codimension greater than 1 of any $Q_i$.

For such $a$ and $b$ let $\{a_0,\dots,a_k\}$ consist of $a$, $b$ and all points where $[a,b]$ intersects a facet of some $Q_i$. The points $a_0,\dots,a_k$ are ordered by distance from $a$, so that $a_0=a$ and $a_k=b$. We see that $f$ is linear on every segment $[a_{i-1},a_i]$ and is convex on every union $[a_{i-1},a_i]\cup [a_i,a_{i+1}]$. However, the lemma is obvious when $Q$ has dimension 1 and by applying the lemma to the segment $[a,b]$ we see that $f$ is convex on $[a,b]$ which implies the desired condition.
\end{proof}

We now fix a face $F$ of $K$ and a point $w\in C(I^h,I^F)$ for the rest of this section and prove the first statement in part (a) of Theorem~\ref{main}.
\begin{proposition}
The regular subdivision $\Theta_{O(P,<)}(\psi(w))$ is composed of polytopes of the form $O(P,<')$ with $<'$ stronger than $<$.
\end{proposition}
\begin{proof}
Let $f$ be the convex function on $O(P,<)$ that has $f(v_a)=\psi(w)_{v_a}=w_a$ for all $a\in\mL$ and is linear on every polytope in $\Theta_{O(P,<)}(\psi(w))$ (as discussed in Subsection~\ref{secondary}). If a convex polytope is a union of simplices of the form $O(P,\prec)$ with $\prec$ a linearization of $<$, then all of its facets are given by inequalities of the forms $x_p\le x_q$, $x_p\ge 0$ and $x_p\le 1$. This means that any such polytope is necessarily an order polytope and to prove the proposition it suffices to show that $f$ is linear on every simplex $O(P,\prec)$. 

There exists a unique function $g$ on $O(P,<)$ such that $g(v_a)=w_a$ for all $a\in\mL$ and $g$ is linear on every simplex $O(P,\prec)$. We are to show that $f$ and $g$ coincide or, equivalently, that $g$ is convex. By Lemma~\ref{convex} it suffices to check that $g$ is convex on every union of two simplices $O(P,\prec)$ and $O(P,\prec')$ with a common facet.

Thus we have two simplices $O(P,\prec)$ and $O(P,\prec')$ with a common facet and a function $g$ which is linear on each of them. Let $h$ and $h'$ be linear functions on $\bR^P$ such that $h$ and $g$ coincide on $O(P,\prec)$ while $h'$ and $g$ coincide on $O(P,\prec')$. Let $v_a$ be the vertex of $O(P,\prec)$ not contained in $O(P,\prec')$ and $v_{a'}$ be the vertex of $O(P,\prec')$ not contained in $O(P,\prec)$. The convexity of $g$ on $O(P,\prec)\cup O(P,\prec')$ is equivalent to either of the inequalities $h(v_{a'})\ge h'(v_{a'})=w_{a'}$ and $h'(v_a)\ge h(v_a)=w_a$. 

We have seen that $\{a,a'\}$ is a diamond pair and that $a\land a'$ and $a\lor a'$ are contained in $C_\prec\cap C_{\prec'}$. The latter means that the points $v_{a\land a'}$ and $v_{a\lor a'}$ are vertices of both $O(P,\prec)$ and $O(P,\prec')$. Now observe that $v_a+v_{a'}=v_{a\land a'}+v_{a\lor a'}$. We deduce 
\begin{equation}\label{ineq}
h(v_{a'})=h(v_{a\land a'})+h(v_{a\lor a'})-h(v_a)=w_{a\land a'}+w_{a\lor a'}-w_a\ge w_{a'}
\end{equation}
where the last inequality follows from Theorem~\ref{hibimonfacets}.
\end{proof}

\begin{remark}
Consider the graph $\Gamma$ the vertices in which are linearizations of $<$ with $\prec$ and $\prec'$ adjacent when $O(P,\prec)$ and $O(P,\prec')$ have a common facet. This is sometimes referred to as the \emph{adjacency graph} of the triangulation formed by the simplices $O(P,\prec)$. Let $\prec$ and $\prec'$ be linearizations given by $p_1\prec\dots\prec p_{|P|}$ and $p'_1\prec'\dots\prec' p'_{|P|}$. It is not hard to see that $O(P,\prec)$ and $O(P,\prec')$ have a common facet if and only if the tuples $(p_1,\dots,p_{|P|})$ and $(p'_1,\dots,p'_{|P|})$ are obtained from each other by exchanging two consecutive elements. If one views these tuples as permutations of the set $P$, then the above condition means that the two permutations are obtained from each other via multiplication by an elementary transposition. This shows that $\Gamma$ is a full subgraph of the Hasse diagram of the weak Bruhat order on permutations. The latter graph can also be represented as the 1-skeleton of the permutahedron. This observation is closely related to Remark~\ref{permutahedra}.
\end{remark}

We have obtained orders $<_1,\dots,<_m$ on $P$ such that \[\Theta_{O(P,<)}(\psi(w))=\{O(P,<_1),\dots,O(P,<_m)\}.\] Our next goal is to prove that $Y_F$ is the union of the Hibi varieties $H(\mM_{<_i})\subset\bP$. This will imply part (b) and will also complete the proof of part (a), since part (b) is seen to imply that the subdivision is independent of $w\in C(I^h,I^F)$. Because $H(\mM_{<_i})$ is a Hibi variety, its vanishing ideal $I_i$ is described as follows. $I_i$ is generated by the expressions $d(a,b)$ for incomparable $a,b\in\mM_{<_i}$ and the variables $X_c$ for $c\notin\mM_{<_i}$. Denote $\widetilde I^F=\bigcap_i I_i$, we are to prove that $\widetilde I^F=I^F$.

For each $i\in[1,m]$ consider a polynomial ring $S_i=\bC[\{z_{p,i},p\in P\},t_i]$. Let $\varphi_i^h$ be the map from $R$ to $S_i$ with \[\varphi_i^h(X_a)=t_i\prod_{p\in\iota(a)}z_{p,i}.\] By Proposition~\ref{kernel}, the kernel of $\varphi_i^h$ is $I^h$. Consequently, the kernel of the map $\varphi^h=\bigoplus_i\varphi_i^h$ from $R$ to $S^F=\bigoplus_i S_i$ is also $I^h$. 

Furthermore, for $i\in[1,m]$ let $\varphi_i^F$ be the map from $R$ to $S_i$ with 
\begin{equation}\label{phiF}
\varphi_i^F(X_a)=t_i\prod_{p\in\iota(a)}z_{p,i}
\end{equation}
if $v_a\in O(P,<_i)$ (equivalently, $a\in\mM_{<_i}$) and $\varphi_i^F(X_a)=0$ otherwise. In view of Proposition~\ref{kernel} and the above description of $I_i$, the kernel of $\varphi_i$ is $I_i$. Consequently, the kernel of the map $\varphi^F=\bigoplus_i\varphi_i^F$ from $R$ to $S^F$ is $\widetilde I^F$.

\begin{proposition}\label{contains}
We have $I^F\subset\widetilde  I^F$.
\end{proposition}
\begin{proof}
First suppose that $w_{a_0}=0$ for the minimal element $a_0$ of $\mL$. Let $f$ be the convex function with $f(v_a)=w_a$ that is linear on every $O(P,<_i)$. Since $0=v_{a_0}$ is a vertex of every $O(P,<_i)$ and $f(v_{a_0})=0$, we have vectors $\alpha_i\in\bR^P$ such that $f(x)=(\alpha_i,x)$ for $x\in O(P,<_i)$.

Consider the $\bR$-grading $\grad$ on $R$ given by $\grad X_a=w_a$. If we decompose a polynomial $r\in R$ into a sum of nonzero $\grad$-homogeneous polynomials, then the summand with the least grading will equal $\initial_w r$. We also define an $\bR$-grading $\grad_S$ on $S^F$ by setting $\grad_S z_{p,i}=(\alpha_i)_p$ and $\grad_S t_i=0$. For any $a\in\mL$ and any $i\in[1,m]$ we have $\grad_S\varphi_i^h(X_a)=(\alpha_i,v_a)$, in particular, $\grad_S\varphi_i^h(X_a)=w_a$ when $v_a\in O(P,<_i)$. Furthermore, the convexity of $f$ implies that if $v_a\in O(P,<_i)$ and $v_a\notin O(P,<_j)$, then \[w_a=(\alpha_i,v_a)<(\alpha_j,v_a)=\grad_S\varphi_j^h(X_a).\] We see that for any $a$ and $i$ either $\grad_S\varphi_i^h(X_a)=w_a$ and $\varphi_i^F(X_a)=\varphi_i^h(X_a)$ or $\grad_S\varphi_i^h(X_a)>w_a$ and $\varphi_i^F(X_a)=0$. In other words, one could say that $\varphi^F(X_a)$ is the initial part of $\varphi^h(X_a)$ with respect to $\grad_S$. This statement immediately generalizes to any monomial $M$ in the $X_a$: for any $a$ and $i$ either $\grad_S\varphi_i^h(M)=\grad M$ and $\varphi_i^F(M)=\varphi_i^h(M)$ or $\grad_S\varphi_i^h(M)>\grad M$ and $\varphi_i^F(M)=0$. Consequently, the $\grad_S$-homogeneous component of $\varphi^h(M)$ of grading $\grad M$ is equal to $\varphi^F(M)$ while the $\grad_S$-homogeneous component of any grading $\beta<\grad M$ is 0. 

Consider a polynomial $r\in I^h$, denote $\grad(\initial_w r)=\gamma$. Since $\varphi^h(r)=0$, the $\grad_S$-homogeneous component of $\varphi^h(r)$ of grading $\gamma$ is also zero. However, the above shows that the only monomials $M$ occurring in $r$ for which $\varphi^h(M)$ can have a nonzero $\grad_S$-homogeneous component of grading $\gamma$ are those occurring in $\initial_w r$. As a result, the $\grad_S$-homogeneous component of $\varphi^h(r)$ of grading $\gamma$ is equal to $\varphi^F(\initial_w r)$. Hence $\varphi^F(\initial_w r)=0$ and $I^F=\initial_w I^h$ is contained in the kernel of $\varphi^F$ which is $\widetilde I^F$.

Now, if $w_{a_0}\neq 0$, consider $w'$ given by $w'_a=w_a-w_{a_0}$. Then $w'_{a_0}=0$ and $\Theta_{O(P,<)}(\psi(w))=\Theta_{O(P,<)}(\psi(w'))$ which reduces the proposition to the above case.
\end{proof}

To prove that the ideals coincide we will need the following general property of polyhedral subdivisions.
\begin{lemma}\label{samesum}
Consider equisized tuples $(a_1,\dots,a_k)$ and $(b_1,\dots,b_k)$ of elements of $\mL$ such that $\sum v_{a_j}=\sum v_{b_j}$. Suppose that for some $i_1,i_2\in[1,m]$ all $v_{a_j}$ are vertices of $O(P,<_{i_1})$ and all $v_{b_j}$ are vertices of $O(P,<_{i_2})$. Then all $v_{a_j}$ and all $v_{b_j}$ are contained in $O(P,<_{i_1})\cap O(P,<_{i_2})$.
\end{lemma}
\begin{proof}
The point $x=\sum v_{a_j}/k=\sum v_{b_j}/k$ is contained in $O(P,<_{i_1})\cap O(P,<_{i_2})$ as a convex linear combination of vertices of each of the polytopes. Recall that all faces of all $O(P,<_i)$ form a polyhedral complex, let $Q$ be the minimal polytope in this complex which contains $x$. Evidently, $Q$ is a face of both $O(P,<_{i_1})$ and $O(P,<_{i_2})$. We claim that all $v_{a_j}$ and all $v_{b_j}$ lie in $Q$. Indeed, there exists a linear function $h$ on $\bR^P$ such that $h(y)=0$ for any $y\in Q$ and $h(y)>0$ for any other $y\in O(P,<_{i_1})$. In particular, $h(x)=0$ but $h(x)$ is also the average of the of the nonnegative values $h(v_{a_j})$, hence all $h(v_{a_j})=0$. The proof that $Q$ contains all $v_{b_j}$ is similar.
\end{proof}

\begin{proposition}
We have $I^F=\widetilde I^F$.
\end{proposition}
\begin{proof}
We introduce another pair of gradings on $R$ and $S^F$. Let the $\bZ$-grading $\deg$ on $R$ be the total degree in all $X_a$. The ideals $I^F$ and $\widetilde I^F$ are $\deg$-homogeneous. Let the $\bZ$-grading $\deg_S$ on $S^F$ be the total degree in all $t_i$. For a $\deg$-homogeneous polynomial $r\in R$ we have $\deg_S\varphi^h(r)=\deg_S\varphi^F(r)=\deg r$. For a $\deg$-homogeneous subspace $U\subset R$ and $l\in\bZ$ we denote $U_l$ the $\deg$-homogeneous component of $U$ of grading $l$. Similarly, for a $\deg_S$-homogeneous subspace $T\subset S^F$ and $l\in\bZ$ we denote $T_l$ the $\deg_S$-homogeneous component of $T$ of grading $l$.

In view of Proposition~\ref{contains}, to show that $I^F=\widetilde I^F$ it suffices to show that $\dim I^F_l=\dim\widetilde I^F_l$ for all $l$. Since $I^F$ is an initial ideal of $I^h$ we have $\dim I^F_l=\dim I^h_l$ for all $l$. Therefore, in view of $\ker \varphi^h=I^h$ and $\ker\varphi^F=\widetilde I^F$, it suffices to show that $\dim\varphi^F(R)_l=\dim\varphi^h(R)_l$ for all $l$.

Choose $l\in\bZ_{\ge 0}$. Consider the monomial $M=X_{a_1}\dots X_{a_l}$ and the point $x=v_{a_1}+\dots+v_{a_l}$. Then $\varphi^h(M)$ is determined by the point $x$: every $\varphi_i^h(M)$ is equal to $t_i^l\prod_{p\in P} z_{p,i}^{x_p}$. Therefore, $\dim\varphi^h(R)_l$ is equal to the cardinality of the $l$-fold Minkowski sum $V_l=V+\dots+V$.

We are to show that $\dim\varphi^F(R)_l$ is also equal to the cardinality of $V_l$. First, let us show that for every $x\in V_l$ we may choose a monomial a $M(x)=X_{a_1(x)}\dots X_{a_l(x)}$ such that $x=v_{a_1(x)}+\dots+v_{a_l(x)}$ and $\varphi^F(M(x))\neq 0$. Note that $\varphi^F_i(M(x))\neq 0$ if $O(P,<_i)$ contains all $v_{a_j(x)}$. Suppose that there exist $a_1\le\dots\le a_l$ such that $x=v_{a_1}+\dots+v_{a_l}$. Then there exists a maximal chain $C$ in $\mL$ containing all the $a_j$. However, by part (b) of Proposition~\ref{maxchains}, $C=C_\prec$ for some linearization $\prec$ of $<$ and all the $v_{a_j}$ are vertices of $O(P,\prec)$. The simplex $O(P,\prec)$ is contained in some $O(P,<_i)$ and $\varphi^h_i(M_x)\neq 0$ so we may set $a_j(x)=a_j$.

Now we show that for every $x\in V_l$ there exist elements $a_1\le\dots\le a_l$ with $x=v_{a_1}+\dots+v_{a_l}$. Choose any $a^0_1,\dots,a^0_l$ with $x=v_{a^0_1}+\dots+v_{a^0_l}$. If some $a^0_j$ and $a^0_k$ are incomparable, then let the tuple $(a^1_1,\dots,a^1_l)$ be obtained from $(a^0_1,\dots,a^0_l)$ by replacing $a^0_j$ and $a^0_k$ with $a^0_j\land a^0_k$ and $a^0_j\lor a^0_k$. Note that $v_{a^1_1}+\dots+v_{a^1_l}=x$ and \[|a^0_i|^2+|a^0_j|^2<|a^0_j\land a^0_k|^2+|a^0_j\lor a^0_k|^2.\] Consequently, after a finite number of such replacements we will obtain a tuple $(a^N_1,\dots,a^N_l)$ with $v_{a^N_1}+\dots+v_{a^N_l}=x$ in which every two elements are comparable. Such a tuple may be reordered to obtain a weakly increasing tuple $a_1\le\dots\le a_l$.

For the chosen monomials $M(x)$ the images $\varphi^F(M(x))$ are seen to be linearly independent. We are left to prove that they span $\varphi^F(R)_l$. To do so we consider a monomial $M=X_{a_1}\dots X_{a_l}$ with $x=v_{a_1}+\dots+v_{a_l}$ and show that either $\varphi^F(M)=0$ or $\varphi^F(M)=\varphi^F(M(x))$. Indeed, suppose that $\varphi^F(M)\neq0$. Then we claim that $\varphi_i^F(M)=\varphi_i^F(M(x))$ for all $i$, i.e.\ the set of $O(P,<_i)$ which contain all $v_{a_j}$ coincides with the set of $O(P,<_i)$ which contain all $v_{a_j(x)}$. This, however, follows directly from Lemma~\ref{samesum}.
\end{proof}

\begin{remark}\label{nonregular}
Consider orders $<'_1,\dots,<'_{m'}$ such that the $O(P,<'_i)$ form a polyhedral but \emph{non-regular} subdivision of $O(P,<)$. Let $I'_i$ be the vanishing ideal of $H(\mM_{<'_i})\subset\bP$ and let $\widetilde I'$ be the intersection of these vanishing ideals. One can use the above argument to show that $I^m$ is still an initial ideal $\widetilde I'$ in this generality, i.e.\ the union of the $H(\mM_{<'_i})$ still degenerates into $Y^K$ (however, this union is not a degeneration of $H$). Such polyhedral but non-regular subdivisions can be shown to exist.
\end{remark}

Theorem~\ref{main} is now easily deduced.
\begin{proof}[Proof of Theorem~\ref{main}]
In the above setting we have seen that $Y^F$ is the union of the toric varieties $H(\mM_{<_i})$. All of these varieties have dimension $|P|=\dim Y^F$, hence they are the irreducible components of $Y^F$. This shows that the subdivision $\Theta_{O(P,<)}(\psi(w))$ is independent of $w\in C(I^h,I^F)$. We set $m(F)=m$ and $<_i^F=<_i$ and obtain parts (a) and (b) of the Theorem.

We are left to show that every regular subdivision of $O(P,<)$ that consists of polytopes of the form $O(P,<')$ has the form $\Theta_{O(P,<)}(u)$ for some $u\in\psi(K)$, this will prove part (c). In view of part (a) of the theorem, the subdivision $\Theta_{O(P,<)}(u)$ is the same for all points $u$ in the interior of $\psi(K)$. Moreover, if $u$ lies on the boundary of $\psi(K)$, then the subdivision $\Theta_{O(P,<)}(u)$ is different then that obtained for an interior point, since $\psi^{-1}(u)$ is contained in a proper face of $K$. Together with the fact that $\dim\psi(K)=\dim K=|V|$ this shows that $\psi(K)$ is a maximal cone in the secondary fan of $O(P,<)$, in view of part (a) of Theorem~\ref{secondaryfan}.

By part (c) of Theorem~\ref{secondaryfan}, the subdivision $\Theta_{O(P,<)}(u)$ obtained for any $u$ in the interior of $\psi(K)$ is a triangulation. However, this triangulation consists of the polytopes $O(P,<_i^K)$. Since they must be simplices, the orders $<_1^K,\dots,<_{m(K)}^K$ are precisely the linearizations of $<$. By part (b) of Theorem~\ref{secondaryfan}, every regular subdivision that coarsens this triangulation is obtained as $\Theta_{O(P,<)}(u)$ for some $u$ on the boundary of $\psi(K)$, in other words, it equals $\{O(P,<_1^F),\dots,O(P,<_{m(F)}^F)\}$ for some face $F$ of $K$. However, we have seen that a subpolytope of $O(P,<)$ has the form form $O(P,<')$ if and only if it is a union of the simplices $O(P,\prec)$. Therefore, a regular subdivision of $O(P,<)$ coarsens the triangulation formed by the simplices $O(P,\prec)$ if and only if it consists of polytopes of the form $O(P,<')$. 
\end{proof}

\begin{cor}\label{maxcone}
The cone $\psi(C(I^h,I^m))$ is a maximal cone in the secondary fan of $O(P,<)$.
\end{cor}

It is known that the secondary fan of a polytope is refined by the Gr\"obner fan of the corresponding toric ideal, see~\cite[Proposition 8.15]{S}. The above corollary shows that they, in fact, share at least one maximal cone when the polytope in question is an order polytope.

\begin{remark}\label{permutahedra}
Let $\mL$ be the lattice of all subsets in $\{1,\dots,n\}$ with intersection as meet, union as join and inclusion as the order relation. Functions $u:\mL\to\bR$ satisfying \[u(a)+u(b)\ge u(a\cap b)+u(a\cup b)\] are known as \emph{submodular functions}. Evidently, the set of submodular functions can be viewed as polyhedral cone in $\bR^\mL$ (known as the~\emph{submodular cone}) which coincides with $-K$. Submodular functions are in correspondence (\cite[Proposition 12]{MPSSW}) with an important family of polytopes known as \emph{generalized permutahedra} (\cite{P}). In fact, this correspondence is easily expressed in terms of the above construction. For a submodular function $u$ suppose that $-u=w\in K$. As above, let $f$ be the function on $O(P,<)$ (the unit cube) with all $f(v_a)=w_a$ and linear on every $O(P,<_i^F)$. Then the restriction of $f$ to $O(P,<_i^F)$ has the form $(\alpha_i,x)+w_{\varnothing}$. The convex hull of the points $-\alpha_i$ is the generalized permutahedron corresponding to $u$. Furthermore, using the notions of \emph{extended submodular functions} and \emph{extended generalized permutahedra} (see~\cite[Subsection 12.4]{AA}) one could generalize this observation to an arbitrary lattice $\mL$. It would be interesting to explore the relationship between the algebraic and geometric objects studied in this paper and the corresponding generalized permutahedra and associated combinatorial objects.
\end{remark}

\section{Semitoric subdivisions from weight polytopes}\label{weightpolytopes}

In this section we will explain an alternative method of describing the variety $Y^F$ and the corresponding subdivision of the order polytope. We will start with a few simple Gr\"obner theoretic observations that apply equally well to arbitrary ideals in polynomial rings.  

As already mentioned, every point $w\in\bR^\mL$ defines a grading $\grad_w$ on $R$ by $\grad_w(X_a)=w_a$. 
\begin{proposition}
For an ideal $I\subset R$ the set $U(I)\subset\bR^\mL$ consisting of $w\in\bR^\mL$ for which $I$ is $\grad_w$-homogeneous is a linear subspace.
\end{proposition}
\begin{proof}
Suppose that $I$ is homogeneous with respect to $\grad_{w_1}$ and $\grad_{w_2}$. Then any $p\in I$ decomposes as $p=p_1+\dots+p_k$ where the $p_i$ are $\grad_{w_1}$-homogeneous and lie in $I$. Every $p_i$ further decomposes into a sum of $\grad_{w_2}$-homogeneous polynomials that lie in $I$. This expresses $p$ as sum of polynomials that are homogeneous with respect to both $\grad_{w_1}$ and $\grad_{w_2}$ and, therefore, any $\grad_{\alpha_1w_1+\alpha_2w_2}$.
\end{proof}

\begin{proposition}
Let ideals $I,I'\subset R$ be such that $C(I,I')$ is nonempty. Then $U(I')$ is the linear span of $C(I,I')$.
\end{proposition}
\begin{proof}
Every $w\in C(I,I')$ satisfies $\initial_w I=I'$ which implies that $I'$ is $\grad_w$-homogeneous. Therefore, the linear span of $C(I,I')$ is contained in $U(I')$. Now suppose that $w'\in U(I')$ does not lie in the linear span of $C(I,I')$ and consider $w\in C(I,I')$. Note that $\initial_{w'} I'=I'$. By \cite[Proposition 1.13]{S} we have \[\initial_{w'}(\initial_w I)=\initial_{w+\varepsilon w'}I\] for sufficiently small $\varepsilon>0$. However, the left-hand side above is equal to $I'$ which means that $w+\varepsilon w'\in C(I,I')$, this contradicts our choice of $w'$ and $w$. 
\end{proof}

We now define the key new object considered in this section. Every $a\in\mL$ defines a functional $\la_a$ on $\bR^\mL$ by $\la_a(w)=w_a$. For a face $F$ of $K$ let $\la_a^F\in U(I^F)^*$ be the restriction of $\la^a$ to $U(I^F)$ (which is the linear hull of $F$). We define the \emph{weight polytope} $\Pi^F\subset U(I^F)^*$ as the convex hull of the set $\{\la_a^F,a\in\mL\}$.

The terminology is due to the fact that the points $\la_a^F$ are, in fact, the weights of a certain torus representation. The torus $T=(\bC^*)^\mL$ acts on $R$ with a torus point $(t_a,a\in\mL)$ mapping each $X_a$ to $t_aX_a$. The lattice $\Hom(\bC^*,T)$ is naturally isomorphic to the lattice $\bZ^\mL\subset\bR^\mL$. The saturated sublattice $U(I^F)\cap\bZ^\mL$ corresponds to a subtorus $T^F\subset T$ (which is also seen to be the identity component of the stabilizer of $I^F$). The dual lattice $\Hom(T^F,\bC^*)$, the elements of which are variously known as \emph{weights}, \emph{coweights} and \emph{characters}, is naturally embedded into $U(I^F)^*$. Thus we have distinguished a lattice of integer points within the space $U(I^F)^*$. Note that the functionals $\la_a^F$ take integer values on the sublattice $U(I^F)\cap\bZ^\mL$ which lets us view them as weights and makes $\Pi^F$ a lattice polytope.

Recall the grading $\deg$ on $R$ defined in the previous section, in fact, $\deg=\grad_{(1,\dots,1)}$. Since the ideal $I^F$ is $\deg$-homogeneous, $\deg$ induces a grading on the quotient $R/I^F$ which we also denote $\deg$. Since elements of $T^F$ preserve $I^F$, we obtain a $T^F$-action on the variety $Y^F$ and on its homogeneous coordinate ring $R/I^F$. Moreover, we obtain a $T^F$ action on every $\deg$-homogeneous component $(R/I^F)_l$ of $R/I^F$. If we now view $(R/I^F)_1$ as a $T^F$-representation, it decomposes into the direct sum of $|\mL|$ one-dimensional representations: the lines spanned by the images of the $X_a$. Each of these one-dimensional representations corresponds to an element of $\Hom(T^F,\bC^*)$. In fact, by tracing the definitions one sees that the representation given by $X_a$ corresponds to $\la_a^F\in\Hom(T^F,\bC^*)$. This allows us to say that the $\la_a^F$ are the weights of the representation $(R/I^F)_1$ and form its \emph{weight diagram} the convex hull of which is the weight polytope $\Pi^F$.

We establish a few basic properties of the polytopes $\Pi_F$. First of all, note that $I^K=I^m$ and $U(I^K)=\bR^\mL$. Consider the basis in $U(I^K)^*$ dual to the standard basis in $\bR^\mL$, the elements of this basis are the weights $\la_a$. Hence the polytope $\Pi_K$ is the $(|\mL|-1)$-dimensional simplex with these basis vectors as vertices. The polytope obtained in the other extremal case can be identified with $O(P,<)$. Let $F_0$ denote the apex of $K$, then $I^{F_0}=I^h$.

\begin{proposition}\label{hibigradings}
There exists a nondegenerate affine map $\zeta$ from $\bR^P$ to $U(I^h)^*$ taking $v_a$ to $\la_a^{F_0}$. In particular, $\zeta(O(P,<))=\Pi^{F_0}$.   
\end{proposition}
\begin{proof}
First, for a functional $f\in(\bR^P)^*$ and $c\in\bR$ consider the point $w\in\bR^\mL$ with $w_a=f(v_a)+c$ for all $a\in\mL$. It is evident that $w\in U(I^h)$: for any $a,b\in\mL$ we have \[w_a+w_b=w_{a\land b}+w_{a\lor b}\] so that $d(a,b)$ is $\grad_w$-homogeneous. This provides a nondegenerate linear map \[\eta:(\bR^P)^*\oplus\bR\to U(I^h).\] Moreover, by Corollary~\ref{maxcone} and part (d) of Theorem~\ref{secondaryfan} the dimension of $F^0$ is equal to $|P|+1$. This means that $\eta$ identifies $(\bR^P)^*\oplus\bR$ with $U(I^h)$. The dual of the inverse map $(\eta^{-1})^*$ identifies $\bR^P\oplus\bR^*$ with $U(I^h)^*$. This allows us to set $\zeta(x)=(\eta^{-1})^*(x,1)$ for $x\in\bR^P$ where $1\in\bR^*$ is the functional with $1(c)=c$. For $w\in U(I^h)$ we then indeed have \[\zeta(v_a)(w)=(v_a,1)(\eta^{-1}(w))=w_a.\qedhere\] 
\end{proof}

For two faces with $F\subset G$ we have the inclusion $U(I^F)\subset U(I^G)$ and the projection $\pi^G_F$ from $U(I^G)^*$ to $U(I^F)^*$. For any $a$ we have $\pi^G_F(\la_a^G)=\la_a^F$ and, consequently, $\pi^G_F(\Pi^G)=\Pi^F$. 
\begin{proposition}\label{pitrivia}
For any face $F$ of $K$ the following hold.
\begin{enumerate}[label=(\alph*)]
\item Every $\la_a^F$ is a vertex of $\Pi^F$.
\item $\Pi^F$ contains no integer points other than the $\la_a^F$.
\item $\Pi^F$ is a codimension 1 polytope in $U(I^F)^*$, i.e.\ $\dim\Pi^F=\dim F-1$.
\end{enumerate}
\end{proposition}
\begin{proof}
Proposition~\ref{hibigradings} implies that $\Pi^{F_0}$ has precisely $|\mL|$ vertices. Since $\pi^F_{F_0}(\Pi^F)=\Pi^{F_0}$, the polytope $\Pi^F$ has at least as many vertices as $\Pi^{F_0}$. Since $\Pi^F$ is the convex hull of the $|\mL|$ points $\la_a^F$, all of these points must be its vertices, part (a) follows.

First we prove part (b) in the case $F=F_0$. Since $O(P,<)$ contains no integer points other than its vertices, it suffices to show that $\zeta$ identifies $\bZ^P$ with the set of integer point in the affine span of $\Pi^{F_0}$. This, however, follows immediately from the fact that the map $\eta$ considered in the proof of Proposition~\ref{hibigradings} identifies the sets of integer points in $(\bR^P)^*\oplus\bR$ and $U(I^h)$.

Now let $F$ be arbitrary and suppose that an integer point $\mu\in\Pi^F$ is not a vertex. Then we have $\mu=\sum c_a\la_a^F$ where $\sum c_a=1$, all $c_a\ge 0$ and $c_a>0$ for at least two different $a$. On one hand, since $\pi_{F_0}^F(\mu)$ is an integer point it must be equal to some $\la_b^{F_0}$. On the other, $\pi_{F_0}^F(\mu)=\sum c_a\la_a^{F_0}$ and, consequently, $v_b=\sum c_a v_a$. The latter equality, however, implies that $c_b=1$ while all other $c_a=0$ and we reach a contradiction.

To prove part (c) note that, since $\pi^K_F(\Pi^K)=\Pi^F$, the codimension of $\Pi^F$ is no greater than that of the simplex $\Pi^K$ which is 1. Similarly, the codimension of $\Pi^F$ is no less than that of $\Pi^{F_0}$ which is also 1.
\end{proof}

The next theorem will consider the toric variety associated with the polytope $\Pi^F$. Since $\Pi^F$ is not a full-dimensional polytope, this notion needs a brief clarification. Consider an arbitrary lattice polytope $Q\subset\bR^N$. Choose an integer point $x$ in the affine span of $Q$. Then $Q-x$ is a full-dimensional polytope in the real space $\spann(Q-x)$ with lattice of integer points $\spann(Q-x)\cap\bZ^N$. This allows us to define the associated toric variety $Z_Q$ as the toric variety of the normal fan of $Q-x$ in $\spann(Q-x)$. This definition is evidently independent of the choice of $x$. For a discussion of this approach see, for instance,~\cite[Proposition 3.2.9]{CLS} and the preceding paragraph. In particular, the proposition asserts that the torus orbit closure in $Z_Q$ corresponding to a face $Q'$ is isomorphic to $Z_{Q'}$.

We will also be considering a different kind of toric variety corresponding to the lattice polytope $Q$. This is the not necessarily normal projective toric variety $\widehat Z_Q$ defined by the semigroup of integer points in the cone over $Q$. For every integer point $y\in Q\cap\bZ^N$ consider the variable $X'_y$ and let $R'=\bC[\{X'_y,y\in Q\cap\bZ^N\}]$. We also consider the polynomial ring $S=\bC[z_1,\dots,z_N,t]$ and let $I'$ be the kernel of the homomorphism from $R'$ to $S$ mapping $X'_y$ to $t\prod z_i^{y_i}$ where $(y_1,\dots,y_N)$ are the coordinates of $y$. Then $\widehat Z_Q$ is the subvariety in $\bP(\bC^{Q\cap\bZ^N})$ cut out by the ideal $I'$ where we view the $X'_y$ as homogeneous coordinates. If the polytope $Q$ is normal, then $\widehat Z_Q$ is isomorphic to $Z_Q$, more generally, $Z_Q$ is the normalization of $\widehat Z_Q$. For example, when $Q=O(P,<)$, the ring $R'$ can be identified with $R$ and the homomorphism considered above is the map $\varphi$ considered in Proposition~\ref{kernel}. Under this identification $I'$ becomes $I^h$ and $H\simeq\widehat Z_Q$ while Proposition~\ref{toricvariety} holds because the order polytope is normal.

The torus $(\bC^*)^N$ acts on $\widehat Z^Q$ with an open orbit. Moreover, the torus orbits in $\widehat Z^Q$ also correspond to the faces of $Q$. The closure of the orbit corresponding to a face $Q'$ is cut out in $\widehat Z^Q$ by the equations $X'_y=0$ for all $y\notin Q'$. This orbit closure is isomorphic to $\widehat Z_{Q'}$. See~\cite[Appendix 3.A]{CLS} for a treatment of non-normal toric varieties.

Now consider a linearization $\prec$ of $<$ and let the maximal chain $C_\prec\subset\mL$ consist of $a_0,\dots,a_{|P|}$. Then we let $\Delta_\prec$ denote the $|P|$-dimensional face of the simplex $\Pi^K$ with vertices $\la_{a_0}^K,\dots,\la_{a_{|P|}}^K$. Recall that for every face $F$ of $K$ Theorem~\ref{main} provides orders $<_1^F,\dots,<_{m(F)}^F$ on $P$ which enumerate the parts $O(P,<_i^F)$ in the corresponding subdivision of $O(P,<)$ and the irreducible components $H(\mM_{<_i^F})$ of the degeneration $Y^F$. We are ready to prove the main result of this section. 
\begin{theorem}\label{main2}
For any face $F$ of $K$ the following hold.
\begin{enumerate}[label=(\alph*)]
\item The union of the images $\pi_F^K(\Delta_\prec)$ over all linearizations $\prec$ of $<$ is a union of $m(F)$ faces $\Phi_1^F,\dots,\Phi_{m(F)}^F$ of $\Pi^F$, each of these faces has dimension $|P|$. For every $i$ the map $\zeta^{-1}\pi^F_{F_0}$ restricts to a linear bijection from $\Phi_i^F$ to the polytope $O(P,<_i^F)$.
\item In the toric variety $Z_{\Pi^F}$ consider the subvariety $\widetilde Y^F$ that is the union of the torus orbit closures corresponding to the faces $\Phi_1^F,\dots,\Phi_{m(F)}^F$. Then $\widetilde Y^F$ is isomorphic to $Y^F$.
\end{enumerate}
\end{theorem}
\begin{proof}
First observe that in the cases of $F=K$ and $F=F_0$ part (a) would follow from the claim that $\zeta^{-1}\pi^K_{F_0}(\Delta_\prec)=O(P,\prec)$ for all $\prec$. However, for any $a\in\mL$ we have \[\pi^K_{F_0}(\la_a^K)=\la_a^{F_0}=\zeta(v_a).\] The vertices of $\Delta_\prec$ are the $\la_a^K$ with $a\in C_\prec$ while the vertices of $O(P,\prec)$ are the $v_a$ with $a\in C_\prec$ and the claim follows.

For an arbitrary $F$ choose $i\in[1,m(F)]$. Let $\prec_1,\dots,\prec_k$ be those linearizations of $<$ for which $O(P,\prec_j)\subset O(P,<_i^F)$. These are precisely the linearizations of $<_i^F$. We have seen that $\pi^K_{F_0}(\Delta_{\prec_j})=\zeta(O(P,{\prec_j}))$ for all $j$. We also have $\pi_{F_0}^F\pi_F^K=\pi^K_{F_0}$, therefore, the image of the union of $\pi_F^K(\Delta_{\prec_j})$ under $\pi_{F_0}^F$ is $\zeta(O(P,<_i^F))$. To prove part (a) it remains to show that the union of $\pi_F^K(\Delta_{\prec_j})$ over all $j\in[1,k]$ is a face $\Pi^F$, this face will then be $\Phi_i^F$ and have the desired properties. 

A point $w\in U(I^F)$ is a functional on $U(I^F)^*$ taking value $w_a$ on the vertex $\la_a^F$ of $\Pi^F$. Therefore, it suffices to find a $w\in U(I^F)$ such that $w_a=0$ when $v_a$ is a vertex of one of the $O(P,\prec_j)$ and $w_a>0$ otherwise. Evidently, $v_a$ is a vertex of one of the $O(P,\prec_j)$ if and only if it is a vertex of $O(P,<_i^F)$. Consider any $w\in C(I^h,I^F)$, we have the convex function $f$ on $O(P,<)$ which is linear on every $O(P,<_l^F)$ and satisfies $f(v_a)=w_a$. Let $h$ be the linear function on $\bR^P$ such that $h(x)=f(x)$ for $x\in O(P,<_i^F)$. Define $w'\in\bR^\mL$ by $w'_{a}=h(v_a)-f(v_a)$. We have $w'_a=0$ when $a\in O(P,<_i^F)$ and $w'_a>0$ for all other $a$, the latter due to the convexity of $f$.


We move on to the proof of part (b). Before identifying $Y^F$ with the subvariety $\widetilde Y^F\subset Z_{\Pi^F}$ we identify it with a subvariety in $\widehat Z_{\Pi^F}$. By part (b) of Proposition~\ref{pitrivia} the integer points in $\Pi^F$ are parametrized by $\mL$, this lets us view $\widehat Z_{\Pi^F}$ as a subvariety in $\bP$. Moreover, the torus orbit closure in $\widehat Z_{\Pi^F}$ corresponding to the face $\Phi_i^F$ is cut out by the equations $X_a=0$ for $v_a\notin O(P,<_i^F)$. The linear bijection between $\Phi_i^F$ and $O(P,<_i^F)$ then shows that this orbit closure is precisely $H(\mM_{<_i^F})\subset\bP$ and we obtain $Y^F\subset\widehat Z_{\Pi^F}$. 

In particular, we have seen that $H(\mM_{<_i^F})$ is isomorphic to $\widehat Z_{\Phi_i^F}$. However, since $O(P,<_i^F)$ is normal and the linear bijection with $\Phi_i^F$ identifies the sets of integer points in the polytopes (which consist of their vertices), we see that $\Phi_i^F$ is also normal. This means that $\widehat Z_{\Phi_i^F}\simeq Z_{\Phi_i^F}$ which already shows that $\widetilde Y^F$ has the same irreducible components as $Y^F$.

If $\Pi^F$ is a normal polytope, then  $\widehat Z_{\Pi^F}=Z_{\Pi^F}$ and part (b) follows. We are left to consider the case of a non-normal $\Pi^F$. It is known (\cite[Theorem 2.2.12]{CLS}) that there exists an integer $k>0$ such that the dilation $k\Pi^F$ is a normal polytope. $Z_{\Pi_F}$ is then isomorphic to $\widehat Z_{k\Pi^F}$. Let $\Lambda^F=\Hom(T^F,\bC^*)$ be the lattice of integer points in $U(I^F)^*$. Then $Z_{\Pi^F}$ is embedded into the projective space $\bP(\bC^{k\Pi^F\cap\Lambda^F})$.

Above we have obtained an isomorphism between $H(\mM_{<_i^F})\subset Y^F$ and $Z_{\Phi_i^F}\subset\widetilde Y^F$, let us give an explicit description of this isomorphism $\xi_i$ in coordinates. Indeed, let $y\in H(\mM_{<_i^F})$ have homogeneous coordinates $(y_a,a\in\mL)$. Then the point $\xi_i(y)$ with homogeneous coordinates $(\xi_i(y)_\nu,\nu\in k\Pi^F\cap \Lambda^F)$ is defined as follows. First, if $\nu$ is not contained in the face $k\Phi_i^F$ of $k\Pi^F$, then $\xi_i(y)_\nu=0$. Now, since $\Phi_i^F$ is a normal polytope, if $\nu$ lies in $k\Phi_i^F$, then $\nu$ is equal to $\sum_{\la_a^F\in\Phi_i^F} c_a\la_a^F$ for integers $c_a\ge 0$ with $\sum c_a=k$. This allows us to set $\xi_i(y)_\nu=\prod y_a^{c_a}$. The fact that $y\in H(\mM_{<_i^F})$ ensures that $\xi_i(y)_\nu$ is independent of the choice of the $c_a$.

Moreover, it is evident that for a point $y\in H(\mM_{<_i^F})\cap H(\mM_{<_j^F})$ the defined values $\xi_i(y)_\nu$ and $\xi_j(y)_\nu$ coincide for any $\nu$. Therefore, the isomorphisms $\xi_i$ can be ``glued'' into a single bijective morphism $\xi$ from $Y^F$ to $\widetilde Y^F$. However, there exists a canonical torus-equivariant surjective morphism $\rho:Z_{\Pi^F}\to\widehat Z_{\Pi^F}$, see~\cite[Proposition 5.4.7]{CLS}. Since $\rho$ is torus-equivariant, it maps $Z_\Phi$ onto $\widehat Z_\Phi$ for every face $\Phi$ of $\Pi^F$. In particular, its restriction to $Z_{\Phi_i^F}$ is the natural isomorphism to $\widehat Z_{\Phi_i^F}=H(\mM_{<_i^F})$, i.e.\ the inverse of $\xi_i$. This shows that the restriction of $\rho$ to $\widetilde Y^F$ is the inverse of $\xi$ and they are both isomorphisms.
\end{proof}

\begin{remark}
Evidently, the proof of part (b) would be shorter and more conceptual if it were known that $\Pi^F$ is a normal polytope. This is seen to be the case when $F$ is the apex and when $F=K$, it would be interesting to know if this property holds in general.
\end{remark}

\begin{remark}
It can be easily deduced from the above that the map $\zeta^{-1}\pi^F_{F_0}$ restricts to a homeomorphism from the union of the faces $\Phi_i^F$ onto $O(P,<)$. In particular, this union of faces is contractible and also combinatorially equivalent to the subdivision $\{O(P,<_i^F),i\in[1,m(F)]\}$ as a polyhedral complex.
\end{remark}

\section{Degenerations of flag varieties}\label{flags}

In this section we apply the above constructions to obtain semitoric flat degenerations of type A Grassmannians and complete flag varieties. We start with the Grassmannian case which reduces to specializing the above construction.

Choose integers $n\ge 0$ and $k\in[1,n-1]$. Suppose that the lattice $\mL$ is composed of the $n\choose k$ elements $a_{i_1,\dots,i_k}$ with $1\le i_1<\dots<i_k\le n$ where $a_{i_1,\dots,i_k}\le a_{i'_1,\dots,i'_k}$ when $i_j\le i'_j$ for all $j$. One then has \[a_{i_1,\ldots,i_k}\land a_{j_1,\ldots,j_k}=a_{\min(i_1,j_1),\ldots,\min(i_k,j_k)}\] and \[a_{i_1,\ldots,i_k}\lor a_{j_1,\ldots,j_k}=a_{\max(i_1,j_1),\ldots,\max(i_k,j_k)}.\] 

Let $H$, $K$ and other attributes of $\mL$ be as in the above sections. Since the Pl\"ucker coordinates on the Grassmannian $\mathrm{Gr}_k(n)$ are also enumerated by $k$-subsets in $\{1,\dots,n\}$, we may identify the variables $X_a$ with the corresponding Pl\"ucker variables. This lets us view that the Pl\"ucker embedding as a map $\mathrm{Gr}_k(n)\hookrightarrow\bP$ given by the Pl\"ucker ideal $\mathcal I\subset R$. The following is a well known result due to Gonciulea and Lakshmibai. 
\begin{theorem}[\cite{GL}]\label{glgrassmann}
$I^h$ is an initial ideal of $\mathcal I$, in particular, the Hibi variety $H$ is a flat degeneration of $\mathrm{Gr}_k(n)$.
\end{theorem}
We immediately obtain the following.
\begin{cor}
For every face $F$ of $K$ the ideal $I^F$ is an initial ideal of $\mathcal I$ and the semitoric variety $Y^F$ is a flat degeneration of $\mathrm{Gr}_k(n)$. The structure of $Y^F$ is described by Theorems~\ref{main} and~\ref{main2}.
\end{cor}

\begin{remark}
Some of the Grassmannians in other types also have flat degenerations that are Hibi varieties. For instance, the results in~\cite{L} and~\cite{Ca} show that the toric variety associated with a type C Gelfand-Tsetlin polytope is a flat degeneration of the corresponding type C flag variety. However, the type C Gelfand-Tsetlin polytope corresponding to a fundamental weight is seen to be an order polytope which makes the toric degeneration a Hibi variety. We may then apply the results in Sections~\ref{semitoric} and~\ref{weightpolytopes} to obtain a family of semitoric degenerations of any type C Grassmannian.
\end{remark}

We move on to case of the complete flag variety. We consider a different lattice $\mL$. We choose an integer $n\ge 2$ and let $\mL$ consist of elements $a_{i_1,\dots,i_k}$ where $k\in[1,n-1]$ and $1\le i_1<\dots<i_k\le n$. We set $a_{i_1,\dots,i_k}\le a_{i'_1,\dots,i'_l}$ when $k\ge l$ and $i_j\le i'_j$ for $j\in[1,l]$. Note that when $k\ge l$ one has \[a_{i_1,\ldots,i_k}\land a_{i'_1,\dots,i'_l}=a_{\min(i_1,i'_1),\ldots,\min(i_l,i'_l),i_{l+1},\ldots,i_k}\] and \[a_{i_1,\ldots,i_k}\lor a_{i'_1,\ldots,i'_l}=a_{\max(i_1,i'_1),\ldots,\max(i_l,i'_l)}.\] 

Again, we will denote the attributes of $\mL$ as in the above sections, we will also use the shorthand notation $X_{i_1,\dots,i_k}=X_{a_{i_1,\dots,i_k}}$ for the variables in $R$. Note that $R$ can be viewed as the multihomogeneous coordinate ring of the product \[\bP_\mult=\bP(\wedge^1\bC^n)\times\dots\times\bP(\wedge^{n-1}\bC^n)\] where the variable $X_{i_1,\dots,i_k}$ corresponds to the basis vector $e_{i_1}\wedge\dots\wedge e_{i_k}\in\wedge^k\bC^n$. We will call an ideal $I\subset R$ \emph{multihomogeneous} if for each $k\in[1,n-1]$ it is homogeneous with respect to total degree in the variables $X_{i_1,\dots,i_k}$. Every multihomogeneous ideal cuts out a subvariety in $\bP_\mult$. It is easily seen that $I^h$ is multihomogeneous and, consequently, so is $I^F$ for every face $F$ of $K$. We denote by $H_\mult$ the subvariety in $\bP_\mult$ cut out by $I^h$ and by $Y^F_\mult$ the subvariety cut out by $I^F$. Now, recall that the Pl\"ucker embedding of the variety $\mathcal F_n$ of complete flags realizes $\mathcal F_n$ as a subvariety in $\bP_\mult$ and is given by the Pl\"ucker ideal $\mathcal I\subset R$.
\begin{theorem}[\cite{GL}]\label{glflags}
$I^h$ is an initial ideal of $\mathcal I$ and the variety $H_\mult$ is a flat degeneration of $\mathcal F_n$.
\end{theorem}
\begin{cor}
For every face $F$ of $K$ the variety $Y^F_\mult$ is a flat degeneration of $\mathcal F_n$.
\end{cor}
\begin{proof}
In view of the above theorem $I^F$ is an initial ideal of $\mathcal I$. By Theorem~\ref{flatfamily} the algebra $R/I^F$ is a flat degeneration of the algebra $R/\mathcal I$. To obtain a flat family of multiprojective varieties over $\bC$ with $\mathcal F_n$ as the general fiber and $Y^F_\mult$ as the special fiber it remains to show that $R/I^F$ is the multihomogeneous coordinate ring of $Y^F_\mult$. This means showing that $I^F$ is the entire vanishing ideal of $Y^F_\mult$. However, by a multiprojective version of Hilbert's Nullstellensatz, $I^F$ is the vanishing ideal of its zero set if and only if it is radical and saturated with respect to the ``irrelevant'' ideals $\mathfrak M_k=\langle \{X_{i_1,\dots,i_k}\}\rangle$ for all $k$ (see Remark~\ref{hilbert}). The latter condition can be written as $(I^F:\mathfrak M_k)=I^F$ where $(I^F:\mathfrak M_k)$ is the ideal $\{r\in R| r\mathfrak M_k\subset I^F\}$. We have already seen that $I^F$ is radical while the saturatedness can be checked as follows. 

In Section~\ref{semitoric} we described $I^F$ as the intersection of prime ideals $I_1,\dots, I_{m(F)}$. Suppose that $(I^F:\mathfrak M_k)\neq I^F$ for some $k$, then for some $r\in R$ with $r\notin I^F$ we have $r\mathfrak M_k\subset I^F$. We must have $r\notin I_i$ for some $i$ and, in view of $I_i$ being prime, all $X_{i_1,\dots,i_k}\in I_i$. By the definition of the ideals $I_i$ this means that all $a_{i_1,\dots,i_k}\notin\mM_{<_i^F}$. However, $\mM_{<_i^F}$ contains at least one chain which is maximal in $\mL$ and such a chain is seen to contain at least one element of the form $a_{i_1,\dots,i_k}$.
\end{proof}

\begin{remark}\label{hilbert}
Somewhat surprisingly, the authors were unable to find this ``multiprojective Nullstellensatz'' in the literature, see~\cite{MO}. A short proof can be obtained by looking at the affine cones. Indeed, for a multihomogeneous ideal $I\subset R$ let $V(I)\subset\bP_\mult$ bet its zero set in $\bP_\mult$ and let $V_{\mathrm{aff}}(I)$ be its zero set in the affine space $\mathbb A=\wedge^1\bC^n\oplus\dots\oplus\wedge^{n-1}\bC^n$. Let $\mathbb A_k\subset\mathbb A$ be the coordinate subspace where all $X_{i_1,\dots,i_k}=0$. Then for multihomogeneous ideals $I,J\subset R$ we have $V(I)=V(J)$ if and only if \[V_{\mathrm{aff}}(I)\backslash(\mathbb A_1\cup\dots\cup \mathbb A_k)=V_{\mathrm{aff}}(J)\backslash(\mathbb A_1\cup\dots\cup\mathbb A_k).\] Therefore, $I$ is the vanishing ideal of $V(I)$ if and only if $V_{\mathrm{aff}}(I)$ coincides with the closure $\overline{V_{\mathrm{aff}}(I)\backslash(\mathbb A_1\cup\dots\cup \mathbb A_k)}$. The given condition now follows via standard affine algebraic geometry.
\end{remark}

Having established that the varieties $Y^F_\mult$ are flat degenerations of the flag variety, we next aim to show that these varieties are semitoric and describe them with an analog of Theorem~\ref{main}. We start by describing the poset $(P,<)$.

Consider the set $\widetilde P$ composed of elements $p_{r,s}$ where $1\le r\le s\le n$ and $(r,s)$ is not one of $(1,1)$ and $(n,n)$. Define a partial order  $\widetilde<$ on $\widetilde P$ by setting $p_{r,s}\widetilde\le p_{u,v}$ when $r\le u$ and $s\le v$. The poset $(\widetilde P,\widetilde<)$ is easy to visualize, for instance, when $n=4$ its Hasse diagram looks like this:
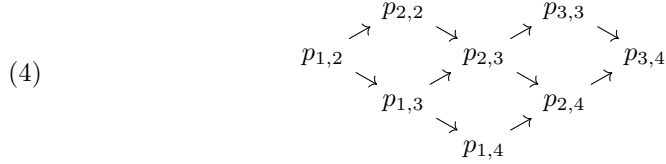
\begin{equation}
\begin{tikzcd}[row sep=1mm,column sep=tiny]\label{n=4}
&p_{2,2}\arrow[rd]&&p_{3,3}\arrow[rd]\\
p_{1,2}\arrow[rd]\arrow[ru]&&p_{2,3}\arrow[rd]\arrow[ru]&&p_{3,4}\\
&p_{1,3}\arrow[rd]\arrow[ru]&&p_{2,4}\arrow[ru]\\
&&p_{1,4}\arrow[ru]
\end{tikzcd}
\end{equation}
\begin{proposition}\label{gtposet}
The poset $(P,<)$ is isomorphic to $(\widetilde P,\widetilde<)$.
\end{proposition}
\begin{proof}
Although this fact can be found in~\cite[Proposition 6.1]{M}, we give a more direct proof by identifying the elements of $\mL$ with the order ideals in $(\widetilde P,\widetilde<)$. Indeed, define a map $\phi:\mL\to\mathcal J(\widetilde P,\widetilde<)$ by the following rule. The set $\phi(a_{i_1,\dots,i_k})$ contains the $i_j-j$ elements $p_{1,n-j+1},\dots,p_{i_j-j,n-j+1}$ for each $j\in[1,k]$ as well as all $p_{r,s}$ with $s<n-k+1$. Then one sees that $\phi(a_{i_1,\dots,i_k})$ is indeed an order ideal and that $\phi$ is an isomorphism between the lattices $\mL$ and $\mathcal J(\widetilde P,\widetilde<)$.
\end{proof}

We will identify $(P,<)$ with $(\widetilde P,\widetilde<)$ viewing the $p_{r,s}$ as elements of $P$. Furthermore, we will assume that the map $\iota$ coincides with the map $\phi$ from the above proof. We will not explicitly identify the $p_{r,s}$ with join-irreducible elements in $\mL$ although such an identification can be found in~\cite{M}. For a point $x\in\bR^P$ we will denote its coordinate corresponding to $p_{r,s}$ by $x_{r,s}$ for short.


Consider an order $<'$ on $P$ which is stronger than $<$. We define the polytope $O_\mult(P,<')\subset O(P,<')$ consisting of points $x$ for which $x_{r,r}=\frac{n-r}{n-1}$. In other words, in terms of the visualization in~\eqref{n=4} we can describe $O_\mult(P,<')$ as the section of $O(P,<')$ by the $n\choose 2$-dimensional subspace where the coordinates corresponding to the top row are equal to $\frac{n-2}{n-1},\dots,\frac1{n-1}$. 

$O_\mult(P,<)$ can be identified with a \emph{marked order polytope}, a notion due to~\cite{ABS}. Consider an arbitrary poset $(\mP,\ll)$ with a subset $M\subset \mP$ of \emph{marked elements} which contains all maximal and minimal elements. For $\la\in\bR^M$ the corresponding marked order polytope $O_{M,\la}(\mP,\ll)\subset\bR^{\mP}$ consists of $x\in\bR^{\mP}$ for which $x_p=\la_p$ for all $p\in M$ and $x_p\ge x_q$ whenever $p\ll q$. In particular, let the set $\overline P$ be obtained from $P$ by adding elements $p_{1,1}$ and $p_{n,n}$. We extend any order $\ll$ on $P$ to $\overline P$ by making $p_{1,1}$ the unique minimal element and $p_{n,n}$ the unique maximal element. Let the subset $M\subset\overline P$ of marked elements consist of $p_{r,r}$ with $r\in[1,n]$. Identify $\bR^P$ with the affine subspace in $\bR^{\overline P}$ consisting of $x$ with $x_{1,1}=1$ and $x_{n,n}=0$. Then $O_\mult(P,<')$ is identified with $O_{M,\mu}(\overline P,<')$ where $\mu\in\bR^M$ is given by $\mu_{p_{r,r}}=\frac{n-r}{n-1}$.

Now recall that we have the inclusion $\mathcal J(P,<')\subset\mathcal J(P,<)$ and the sublattice $\mM_{<'}=\iota^{-1}(\mathcal J(P,<'))$. This lets us view $I^h(\mM_{<'})$ as a subspace of $I^h$ and define $I_{<'}\subset R$ as the ideal generated by $I^h(\mM_{<'})$ and all $X_{i_1,\dots,i_k}$ with $a_{i_1,\dots,i_k}\notin\mM_{<'}$. In other words, this is the vanishing ideal for the embedding $H(\mM_{<'})\subset\bP$ which was discussed in Section~\ref{semitoric}. This ideal is easily seen to be multihomogeneous.

\begin{lemma}\label{multitoric}
The subvariety cut out by $I_{<'}$ in $\bP_\mult$ is isomorphic to the toric variety $Z_{O_{M,\mu}(\overline P,<')}$ associated with the marked order polytope ${O_{M,\mu}(\overline P,<')}$.
\end{lemma}
\begin{proof}
The variety $Z_{O_{M,\mu}(\overline P,<')}$ is isomorphic to $Z=Z_{O_{M,(n-1)\mu}(\overline P,<')}$, since \[O_{M,(n-1)\mu}(\overline P,<')=(n-1)O_{M,\mu}(\overline P,<').\] Furthermore, note that $(n-1)\mu=\mu^1+\dots+\mu^{n-1}$ where $\mu^k\in\bR^M$ is given by $\mu^k_{p_{r,r}}=1$ for $r\le k$ and $\mu^k_{p_{r,r}}=0$ otherwise. This implies that \[O_{M,(n-1)\mu}(\overline P,<')=O_{M,\mu^1}(\overline P,<')+\dots+O_{M,\mu^{n-1}}(\overline P,<').\]

We now make use of the \emph{Minkowski sum property} of marked order polytopes, see~\cite[Theorem 2.5]{FF}. By this property, for any $\la^1,\dots,\la^N\in\bZ^M$ the set of integer points in $O_{M,\la^1+\dots+\la^N}(\overline P,<')$ is the Minowski sum of the sets of integer points in the polytopes $O_{M,\la^i}(\overline P,<')$. This has two implications for us. First, by setting $\la^1=\dots=\la^N$ we see that every $O_{M,\la}(\overline P,<')$ with $\la\in\bZ^M$ is normal. In particular, $O_{M,(n-1)\mu}(\overline P,<')$ is normal which gives us an embedding of $Z$ into $\bP(\bC^\Xi)$ where $\Xi=O_{M,(n-1)\mu}(\overline P,<')\cap\bZ^{\overline P}$. Second, we see that $\Xi$ is equal to the Minkowski sum of the sets $\Xi_k=O_{M,\mu^k}(\overline P,<')\cap\bZ^{\overline P}$ where $k\in[1,n-1]$. This provides an embedding 
\begin{equation}\label{embedding}
\bP(\bC^\Xi)\hookrightarrow\bP(\bC^{\Xi_1}\otimes\dots\otimes\bC^{\Xi_{n-1}})
\end{equation}
as follows. Consider $n-1$ points $x_k\in\Xi_k$, let $e^k_{x_k}\in\bR^{\Xi_k}$ be the corresponding basis vectors and set $e=e^1_{x_1}\otimes\dots\otimes e^{n-1}_{x_{n-1}}$. Consider a point $y\in\bP(\bC^\Xi)$ with homogeneous coordinates $(y_x,x\in\Xi)$. Then the homogeneous coordinate of the image of $y$ under~\eqref{embedding} corresponding to the basis vector $e$ is equal to $y_{x_1+\dots+x_{n-1}}$.

Now recall that $\bR^P$ is embedded into $\bR^{\overline P}$ as the subspace of points $x$ with $x_{1,1}=1$ and $x_{n,n}=0$. We then obviously have the inclusions $O_{M,\mu^k}(\overline P,<')\subset O(P,<')$ for all $k$. We claim that under this inclusion the set of integer points $\Xi_k$ is identified with the set of those integer points in $O(P,<')$ which have the form $v_{a_{i_1,\dots,i_k}}$, i.e. for which the corresponding element in $\mM_{<'}\subset\mL$ has $k$ subscripts. Indeed, the description of the map $\iota$ given in the proof of Proposition~\ref{gtposet} shows that $\iota(a_{i_1,\dots,i_k})$ contains precisely those $p_{r,r}$ for which $r\le n-k$. This means that, with respect to the chosen inclusion $\bR^P\subset\bR^{\overline P}$, the point $v_{a_{i_1,\dots,i_k}}$ is contained in $\Xi_k$ and is not contained in any other $\Xi_l$. This gives us an inclusion of $\bC^{\Xi_k}\subset\wedge^k\bC^n$ where $e^k_{v_{a_{i_1,\dots,i_k}}}\in\bC^{\Xi_k}$ is identified with $e_{i_1}\wedge\dots\wedge e_{i_k}\in\wedge^k\bC^n$.

Combining with~\eqref{embedding} we obtain the chain of embeddings \[Z\hookrightarrow\bP(\bC^\Xi)\hookrightarrow\bP(\bC^{\Xi_1}\otimes\dots\otimes\bC^{\Xi_{n-1}})\hookrightarrow\bP(\wedge^1\bC^n\otimes\dots\otimes\wedge^{n-1}\bC^n).\] Note that, by defintion, the image of the embedding $Z\hookrightarrow\bP(\bC^\Xi)$ is the closure of the image of the torus $(\bC^*)^{\overline P}$ under the map sending $(t_{p_{r,s}},p_{r,s}\in \overline P)$ to the point in $\bP(\bC^\Xi)$ with homogeneous coordinate corresponding to $x\in\Xi$ equal to $\prod t_{p_{r,s}}^{x_{r,s}}$. Combining with the other two inclusions above we see that the image of the embedding of $Z$ into $\bP(\wedge^1\bC^n\otimes\dots\otimes\wedge^{n-1}\bC^n)$ is the image closure $\overline{\xi((\bC^*)^{\overline P})}$ of the following map $\xi$ on the torus. The homogeneous coordinate of $\xi(t_{p_{r,s}},p_{r,s}\in \overline P)$ corresponding to the basis vector \[(e_{i_1^1})\otimes\dots\otimes(e_{i_1^{n-1}}\wedge\dots\wedge e_{i_{n-1}^{n-1}})\] is equal to
\begin{enumerate}[label=(\roman*)]
\item 0 if we have $a_{i_1^k,\dots,i_k^k}\notin\mM_{<'}$ for some $k$, and otherwise
\item $\prod t_{p_{r,s}}^{x_{r,s}}$ where, in terms of the chosen embedding $\bR^P\subset\bR^{\overline P}$, \[x=\sum_{k=1}^{n-1} v_{a_{i_1^k,\dots,i_k^k}}\in\bR^{\overline P}.\]
\end{enumerate}

Having obtained the above realization of $Z$ we move on to the subvariety $Z'\subset\bP_\mult$ cut out by $I_{<'}$. We can express $I_{<'}$ as kernel of a map to a polynomial ring similarly to~\eqref{phiF}. Consider the polynomial ring $S=\bC[\{z_p,p\in P\},t]$. Then $I_{<'}$ is the kernel of the map $\varphi_{<'}$ from $R$ to $S$ with
\begin{equation}
\varphi_{<'}(X_a)=t\prod_{p\in\iota(a)}z_p
\end{equation}
if $a\in\mM_{<'}$ and $\varphi_{<'}(X_a)=0$ otherwise. This means that $Z'$ is the image closure $\overline{\xi'((\bC^*)^P)}\subset\bP_\mult$ where the homogeneous coordinate of $\xi'(t_{p_{r,s}},p_{r,s}\in P)$ corresponding to the basis vector $e_{i_1}\wedge\dots\wedge e_{i_k}$ is equal to $\prod_{p\in\iota(a_{i_1,\dots,i_k})}t_p$ if $a_{i_1,\dots,i_k}\in\mM_{<'}$ and is 0 otherwise. 

Now consider the Segre embedding \[\mathcal S:\bP_\mult\hookrightarrow\bP(\wedge^1\bC^n\otimes\dots\otimes\wedge^{n-1}\bC^n).\] Let $\rho:(\bC^*)^{\overline P}\to (\bC^*)^P$ be the surjective map taking $(t_{p_{r,s}},p_{r,s}\in \overline P)$ to $(t_{p_{r,s}},p_{r,s}\in P)$. One sees that the maps $\mathcal S\xi'\rho$ and $\xi$ from $(\bC^*)^{\overline P}$ to $\bP(\wedge^1\bC^n\otimes\dots\otimes\wedge^{n-1}\bC^n)$ coincide. However, the image closure of the first map is $Z'$ while the image closure of the second is $Z$.
\end{proof}

An important special case of the above lemma is $<'=<$. First we point out that when $\la\in\bR^M$ is such that $\la_{p_{1,1}}>\dots>\la_{p_{n,n}}$, it can be viewed as a strictly dominant $GL_n$ weight. The marked order polytope $O_{M,\la}(\overline P,<)$ is known as the \emph{Gelfand--Tsetlin polytope} (\cite{GT}) corresponding to this strictly dominant weight. All such polytopes are pairwise strongly combinatorially equivalent. Note that $I_<=I^h$, so by applying Lemma~\ref{multitoric} we obtain the following fact which is a well known result due to Kogan and Miller.
\begin{theorem}[\cite{KM}]\label{koganmiller}
The variety $H_\mult$ is isomorphic to the toric variety associated with the Gelfand--Tsetlin polytope $O_{M,\mu}(\overline P,<)$ and, therefore, with any other Gelfand--Tsetlin polytope corresponding to a strictly dominant weight.
\end{theorem}

\begin{remark}
In fact, the polytopes $O_{M,\mu^k}(\overline P,<)$ are also Gelfand--Tsetlin polytopes: those corresponding to the fundamental weights $(1,\dots,1,0,\dots,0)$. These Gelfand--Tsetlin polytopes are easily identified with the order polytopes appearing in the Grassmannian case considered in the beginning of this section.
\end{remark}

Now let $V_{GT}$ denote the set of vertices of the Gelfand--Tsetlin polytope $O_{M,\mu}(\overline P,<\nobreak)$. Since $O_{M,\mu}(\overline P,<)$ is the Minkowski sum of the polytopes $O_{M,\frac{\mu^k}{n-1}}(\overline P,<)$, every $v\in V_{GT}$ is uniquely expressed as $v^1+\dots+v^{n-1}$ with $v^k$ a vertex $O_{M,\frac{\mu^k}{n-1}}(\overline P,<)$. However, in the proof of Lemma~\ref{multitoric} we have obtained a description of the vertices of $O_{M,\mu^k}(\overline P,<)$. Scaling by a factor of $\frac1{n-1}$ we see that the vertices of $O_{M,\frac{\mu^k}{n-1}}(\overline P,<)$ are the $n\choose k$ points $\frac{v_{a_{i_1,\dots,i_k}}}{n-1}$ (under the chosen embedding $\bR^P\subset\bR^{\overline P}$). Thus, given a point $w\in\bR^\mL$ we can define a point $c(w)\in\bR^{V_{GT}}$ by letting the coordinate $c(w)_v$ equal \[w_{a_{i_1^1}}+\dots+w_{a_{i_1^{n-1},\dots,i_{n-1}^{n-1}}}\] where \[v=\frac{v_{a_{i_1^1}}+\dots+v_{a_{i_1^{n-1},\dots,i_{n-1}^{n-1}}}}{n-1}.\] We now prove the analog of Theorem~\ref{main}.

\begin{theorem}\label{mainflags}
For any face $F$ of $K$ the following hold.
\begin{enumerate}[label=(\alph*)]
\item Let $w$ be a point in the relative interior of $F$. The regular subdivision $\Theta_{O_{M,\mu}(\overline P,<)}(c(w))$ of the Gelfand-Tsetlin polytope consists of the polytopes $O_{M,\mu}(\overline P,<_i^F)$ where $i\in[1,m(F)]$ and the orders $<_i^F$ are as in Theorem~\ref{main}.
\item The variety $Y^F_\mult$ is semitoric. It has $m(F)$ irreducible components that are isomorphic to the toric varietes associated with the parts of the subdivision $\Theta_{O_{M,\mu}(\overline P,<)}(c(w))$.
\end{enumerate}
\end{theorem}
\begin{proof}
Note that the subdivision of $O_{M,\mu}(\overline P,<)$ into the polytopes $O_{M,\mu}(\overline P,<_i^F)$ is the section of the subdivision of $O(P,<)$ into the $O(P,<_i^F)$ by the subspace of points $x$ with $x_{r,r}=\frac{n-r}{n-1}$ for $r\in[1,n]$. Consider the function $f$ on $O(P,<)$ which is linear on each $O(P,<_i^F)$ and satisfies $f(v_a)=w_a$ for all $a$. The restriction of $f$ to $O_{M,\mu}(\overline P,<)$ is a convex piecewise linear function with domains of linearity $O_{M,\mu}(\overline P,<_i^F)$. To prove part (a) we show that for every vertex $v\in V_{GT}$ we have $f(v)=\frac{c(w)_v}{n-1}$ (scaling by $\frac1{n-1}$ obviously preserves the subdivision). 

Indeed, consider $v\in V_{GT}$ with \[v=\frac{v_{a_{i_1^1}}+\dots+v_{a_{i_1^{n-1},\dots,i_{n-1}^{n-1}}}}{n-1}.\] The fact that $v$ is a vertex is seen to imply that every coordinate $v_{r,s}$ is equal to some $v_{t,t}=\frac{n-t}{n-1}$ (see, for example,~\cite[Proposition 2.1]{M0}). The set $J_r\subset P$ of $p$ such that $v_p\ge\frac{n-r}{n-1}$ is an order ideal in $P$. We see that \[v=\frac{v_{\iota^{-1}(J_1)}+\dots+v_{\iota^{-1}(J_{n-1})}}{n-1}\] which shows that $a_{i_1^k,\dots,i_k^k}=\iota^{-1}(J_k)$ for all $k$. However, $J_1\subset\dots\subset J_{n-1}$ and, consequently $a_{i_1^1}<\dots<a_{i_1^{n-1},\dots,i_{n-1}^{n-1}}$. The chain $a_{i_1^1},\dots,a_{i_1^{n-1},\dots,i_{n-1}^{n-1}}$ can be extended to a maximal chain $C$ in $\mL$, let $C=C_\prec$ for a linearization $\prec$ of $<$. We see that the simplex $O(P,\prec)$ contains the points $v_{a_{i_1^1}},\dots,v_{a_{i_1^{n-1},\dots,i_{n-1}^{n-1}}}$ as well as their centroid $v$. However, since $f$ is linear on $O(P,\prec)$, we obtain \[f(v)=\frac{f(v_{a_{i_1^1}})+\dots+f\Big(v_{a_{i_1^{n-1},\dots,i_{n-1}^{n-1}}}\Big)}{n-1}=\frac{c(w)_v}{n-1}.\]

To prove part (b) note that, by part (b) of Theorem~\ref{main}, the ideals $I_{<_i^F}$ are the primary components of $I^F$. Therefore, the irreducible components of $Y_\mult^F$ are the zero sets of the $I_{<_i^F}$ in $\bP_\mult$. These are the toric varieties associated with the polytopes $O_{M,\mu}(\overline P,<_i^F)$ by Lemma~\ref{multitoric}.
\end{proof}

\begin{example}
Consider the case $F=K$. The orders $<_i^F$ are then the linearizations of $<$, choose an $i\in[1,m(F)]$. For $k\in[1,n-1]$ let $d_k$ be the number of elements $p\in\overline P$ satisfying $p_{k,k}<_i^F p<_i^F p_{k+1,k+1}$. One then sees that $O_{M,\mu}(\overline P,<_i^F)$ is the product of simplices of dimensions $d_1,\dots,d_{n-1}$, note that $\sum d_k={n\choose 2}$. The above theorem now implies that every irreducible component of the flat degeneration $Y_F^\mult$ is a $n\choose 2$-dimensional product of projective spaces of dimensions $d_1,\dots,d_{n-1}$. 

For instance, by looking at~\eqref{n=4} one may check that there is a total of 12 linearizations when $n=4$. As a result, $Y^F_\mult$ has 12 irreducible components: 8 of the form $\bP^3\times\bP^2\times\bP^1$, 2 of the form $\bP^2\times\bP^2\times\bP^2$ and 2 of the form $\bP^4\times\bP^1\times\bP^1$. It would be interesting to describe the components when $F=K$ for an arbitrary $n$.
\end{example}

To avoid overloading the paper we will not prove an analog of Theorem~\ref{main2} for this setting. Instead we outline the construction omitting the details. For every face $F$ of $K$ a polytope $\Pi^F_\mult\subset U(I^F)^*$ can be defined as the convex hull of all centroids of the form \[\frac{\la_{a_{i_1^1}}+\dots+\la_{a_{i_1^{n-1},\dots,i_{n-1}^{n-1}}}}{n-1}.\] Similarly to $\Pi^F$, this polytope can be interpreted as a weight diagram (scaled by a factor of $\frac1{n-1}$) of a certain homogeneous component of the ring $R/I^F$. 

One sees that for faces $F\subset G$ one has $\pi^F_G(\Pi^G_\mult)=\Pi^F_\mult$, in particular, the map $\zeta^{-1}\pi^F_{F_0}$ is a projection of $\Pi^F_\mult$ onto the Gelfand--Tsetlin polytope $O_\mult(P,<)$. One also deduces that $\Pi^F_\mult$ is, in fact, a codimension $n-2$ section of $\Pi^F$. It is then easily shown that the faces $\Phi_i^F\cap\Pi_\mult^F$ of $\Pi_\mult^F$ are mapped bijectively by $\zeta^{-1}\pi^F_{F_0}$ onto the parts $O_\mult(P,<_i^F)$ of the subdivision of the Gelfand--Tsetlin polytope. This provides an analog of part (a) of Theorem~\ref{main2}. 

An analog of part (b) would state that $Y^F_\mult$ can be embedded into $Z_{\Pi^F_\mult}$ as the union of the orbit closures corresponding to the faces $\Phi_i^F\cap\Pi_\mult^F$. Proving this claim is more involved and requires combining the methods used in the proof of Theorem~\ref{main2} with those used in the proof of Lemma~\ref{multitoric}.

\begin{remark}
One may check that the results obtained in this section can be straightforwardly adapted to a more general multiprojective setting of a marked poset in which the marked elements compose a linearly ordered subset. It would be interesting to find further generalizations that would provide semitoric degenerations of toric varieties associated with arbitrary marked order polytopes.
\end{remark}

\begin{remark}
To expand on the previous remark we point out that there exist several generalizations of order polytopes and marked order polytopes. The most general construction is, likely, that of \emph{marked poset polytopes} found in~\cite{FFLP}. It would be interesting to generalize the above results to produce semitoric degenerations of toric varieties associated with polytopes from this broad family. A special case of particular interest is that of Feigin--Fourier--Littelmann--Vinberg polytopes (\cite{FFL1}), since the corresponding toric varieties are also flat degenerations of flag varieties (\cite{favourable}).
\end{remark}


\begin{thebibliography}{}
\bibitem[AA]{AA}
M. Aguiar, F. Ardila, \emph{Hopf monoids and generalized permutahedra}, \url{https://arxiv.org/abs/1709.07504}

\bibitem[ABS]{ABS} 
F. Ardila, T. Bliem, D. Salazar, {\it Gelfand--Tsetlin polytopes and Feigin--Fourier--Littelmann--Vinberg polytopes as marked poset polytopes},
Journal of Combinatorial Theory, Series A {\bf 118} (2011), no. 8, 2454--2462



\bibitem[BL]{BL}
J. Brown, V. Lakshmibai, \emph{Singular loci of Grassmann-Hibi toric varieties}, The Michigan Mathematical Journal \textbf{59} (2010), no. 2, 243--267. 



\bibitem[Ca]{Ca}
P. Caldero, \emph{Toric degenerations of Schubert varieties}, Transformation Groups \textbf{7} (2002), 51--60




\bibitem[Ch]{Ch}
R. Chiriv\`i, \emph{LS algebras and application to Schubert varieties}, Transformation Groups \textbf{5} (2000), 245--264


\bibitem[CLS]{CLS}
D. Cox, J. Little, H. Schenck, \emph{Toric Varieties}, Graduate Studies in Mathematics, Vol. 124, American Mathematical Society, Providence, 2011.





\bibitem[FF]{FF}
X. Fang, G. Fourier, \textit{Marked chain-order polytopes}, European Journal of Combinatorics \textbf{58} (2016), 267--282

\bibitem[FaFL]{FaFL}
X.~Fang, G.~Fourier, P.~Littelmann, \emph{On toric degenerations of flag varieties}, Representation Theory - Current Trends and Perspectives, EMS Series of Congress Reports, 187--232, 2016

%
%


\bibitem[FFL1]{FFL1} 
E.~Feigin, G.~Fourier, P.~Littelmann,
\emph{PBW filtration and bases for irreducible modules in type ${A}_n$}, Transformation Groups {\bf 16} (2011), 71--89


\bibitem[FFL2]{favourable}
E. Feigin, G. Fourier, P. Littelmann, \emph{Favourable modules: filtrations, polytopes, Newton-Okounkov bodies and flat degenerations}, Transformation Groups {\bf 22} (2017), 321--352

\bibitem[FFLP]{FFLP}
X. Fang, G. Fourier, J.-P. Litza, C. Pegel, \emph{A Continuous Family of Marked Poset Polytopes}, SIAM Journal on Discrete Mathematics \textbf{34} (2020), no. 1, 611--639

\bibitem[FM]{FM}
E.~Feigin, I. Makhlin, \textit{Relative poset polytopes and semitoric degenerations}, \url{https://arxiv.org/abs/2112.05894}




\bibitem[FL]{FL}
X. Fang, P. Littelmann, \emph{From standard monomial theory to semi-toric degenerations via Newton--Okounkov bodies}, Transactions of the Moscow Mathematical Society \textbf{78} (2017), no. 2, 331--356



\bibitem[GKZ]{GKZ} 
I. Gelfand, M. Kapranov, A. Zelevinsky, \textit{Discriminants, Resultants and Multidimensional Determinants}, Birkh\"auser, Boston, 1994

\bibitem[GL]{GL}
N. Gonciulea, V. Lakshmibai, \emph{Degenerations of flag and Schubert varieties to toric varieties}, Transformation Groups {\bf 1} (1996), 215--248

\bibitem[GT]{GT}
I. Gelfand, M. Tsetlin, {\it Finite dimensional representations of the group of 
unimodular matrices}, Doklady Akademii Nauk USSR {\bf 71} (1950), no. 5, 825--828

\bibitem[H]{H}
T. Hibi, \textit{Distributive lattices, affine semigroup rings and algebras with straightening laws}. In
Commutative Algebra and Combinatorics, 93--109, Advanced Studies in Pure Mathematics,
Vol. 11, Amsterdam, 1987


\bibitem[HH]{HH}
J. Herzog, T, Hibi, \emph{Monomial Ideals}, Graduate Texts in Mathematics, Vol. 260, Springer-Verlag, London, 2011






\bibitem[Kn]{Kn}
A. Knutson, \emph{Degenerations of Schubert varieties: a survey}, for the AIM workshop ``Degeneration in algebraic geometry'', 2015, \url{http://aimath.org/~farmer/print/knutson-notes.pdf}.

\bibitem[KM]{KM}
M. Kogan, E. Miller, \emph{Toric degeneration of Schubert varieties and Gelfand–Tsetlin polytopes}, Advances in Mathematics {\bf 193} (2005), no. 1, 1--17



\bibitem[L]{L}
P. Littelmann, \emph{Cones, crystals, and patterns}, Transformation Groups \textbf{3} (1998), 145--179


\bibitem[MO]{MO}
imakhlin, \emph{Reference for the multiprojective Nullstellensatz?}, \url{https://mathoverflow.net/questions/368507/reference-for-the-multiprojective-nullstellensatz}

\bibitem[M0]{M0}
I. Makhlin, \emph{Brion’s theorem for Gelfand–Tsetlin polytopes}, Functional Analysis and Its Applications \textbf{50} (2016), 98--106

\bibitem[M]{M}
I. Makhlin, \emph{Gr\"obner fans of Hibi ideals, generalized Hibi ideals and flag varieties}, \url{https://arxiv.org/abs/2003.02916}




\bibitem[MPSSW]{MPSSW}
J. Morton, L. Pachter, A. Shiu, B. Sturmfels, O. Wienand, \emph{Convex Rank Tests and Semigraphoids}, SIAM Journal on Discrete Mathematics \textbf{23} (2009), no. 3, 1117--1134

\bibitem[MR]{MR}
T. Mora, L. Robbiano, \emph{The Gr\"obner fan of an ideal}, Journal of Symbolic Computation \textbf{6} (1988), no. 2--3, 183--208

\bibitem[MaS]{MLS}
D. MacLagan, B. Sturmfels, {\it Introduction to Tropical Geometry}, Graduate Studies in Mathematics, Vol. 161, American Mathematical Society, Providence, 2015

\bibitem[MS]{MS}
E. Miller, B. Sturmfels, {\it Combinatorial Commutative Algebra}, Graduate Texts in Mathematics, Vol. 227, Springer-Verlag, New York, 2005

\bibitem[MG]{MG}
S. Morier-Genoud, \textit{Geometric Lifting of the Canonical Basis and Semitoric Degenerations of Richardson Varieties}, Transactions of the American Mathematical Society \textbf{360} (2008), no. 1, 215--325


\bibitem[P]{P}
A. Postnikov, \textit{Permutohedra, Associahedra, and Beyond}, International Mathematics Research Notices, \textbf{2009} (2009), no. 6, 1026--1106




\bibitem[St]{stan}
R.~P.~Stanley, \emph{Two poset polytopes}, Discrete \& Computational Geometry \textbf{1} (1986), 9--23

\bibitem[S]{S}
B. Sturmfels, \textit{Gr\"obner Bases and Convex Polytopes}, University Lecture Series, Vol. 8, American Mathematical Society, Providence, 1995


\bibitem[Zh]{Zh}
C.-G. Zhu, \textit{Degenerations of toric ideals and toric varieties}, Journal of Mathematical Analysis and Applications, \textbf{386} (2012), 613--618

\end{thebibliography}
\end{document}